\documentclass[a4paper,11pt]{amsart}
\usepackage{amsfonts}
\usepackage{amssymb}
\usepackage[utf8]{inputenc}
\usepackage{amsmath}
\usepackage{pdflscape}
\usepackage{csquotes}
\usepackage{float}
\usepackage{mathtools}
\usepackage{easyReview}
\usepackage{graphicx}
\usepackage[colorlinks=true, linkcolor=red, citecolor=blue]{hyperref}
\usepackage[]{epsfig}
\usepackage[]{pstricks}
\usepackage{tikz}

\newtheorem{theorem}{Theorem}[section]
\newtheorem{proposition}[theorem]{Proposition}
\newtheorem{lemma}[theorem]{Lemma}

\newtheorem{assumption}[theorem]{Assumption}

\theoremstyle{remark}

\usepackage{mathrsfs}

\numberwithin{equation}{section}
\makeatletter
\@namedef{subjclassname@2010}{\textup{2020} Mathematics Subject Classification}
\makeatother

\begin{document}
	
	\pagenumbering{arabic}	
	\title[ZK system: Stabilization results]{Long-Time Dynamics of the Zakharov--Kuznetsov Equation}
	\author[Capistrano-Filho]{Roberto de A. Capistrano-Filho*}
	\address{Departamento de Matem\'atica,  Universidade Federal de Pernambuco (UFPE), 50740-545, Recife (PE), Brazil.}
	\email{\url{roberto.capistranofilho@ufpe.br}}
	
		\author[Nascimento]{Ailton C. Nascimento}
	\address{Departamento de Matem\'atica, Universidade Federal do Piau\'i (UFPI), 64049-550, Teresina (PI), Brazil.}
	\email{\url{ailton.nascimento@ufpi.edu.br}}
	
	\thanks{*Corresponding author: \url{roberto.capistranofilho@ufpe.br}}
	\thanks{Capistrano–Filho was partially supported by CNPq grant numbers 301744/2025-4, 421573/2023-6, CAPES/COFECUB grant number 88887.879175/2023-00, and PROPG (UFPE) \textit{via} PROAP resources.}
	%\date{Version 2021-12}
	
	\begin{abstract}
This manuscript presents the results of stabilization for the Zakharov--Kuznetsov equation, a two-dimensional Korteweg--de Vries-type equation. We provide rigorous proofs using two different approaches, showing that when a damping mechanism and an internal delay term (anti-damping) are introduced, the solutions of the  Zakharov--Kuznetsov equation exhibit both local and global exponential stabilization properties. A significant contribution of our work is the determination of the optimal constant and the minimal time required to ensure exponential decay of the energy associated with this two-dimensional system.	\end{abstract}

	\maketitle
	
\section{Introduction}
We consider the Zakharov--Kuznetsov (ZK) equation
\begin{equation}\label{ZKeq}
	\partial_t \zeta + \alpha \partial^{3}_{x} \zeta + \gamma\,\partial_{x}\partial_{y}^{2} \zeta +  \zeta\,\partial_x  \zeta = 0,
	\qquad (x,y)\in \mathbb{R}^2,\ t>0,
\end{equation}
where $ \zeta= \zeta(x,y,t)$ is a real-valued function and $\alpha\neq 0$ is a real parameter and $\gamma>0$ constant.  The three-dimensional version of \eqref{ZKeq} was originally derived by Zakharov and Kuznetsov \cite{ZK1974} to describe weakly magnetized ion-acoustic waves in a strongly magnetized plasma. It is important to note that before the previous article, another study, see \cite{ZK1972}, addressed the same physical regime under the assumption that the dispersion associated with the finite Debye ion radius along the magnetic field could be neglected, which made it possible to obtain an exact Kolmogorov-type solution in the weak turbulence regime. We point out also that the two-dimensional form retains physical relevance: it models the propagation of weakly nonlinear ion-acoustic waves in a plasma composed of cold ions and hot isothermal electrons under a uniform magnetic field (see \cite{MoPa-99,MoPa-00} for further details and \cite{LaLiSa-13} for a rigorous mathematical treatment). 

\subsection{Background}
The ZK equation belongs to a broader class of dispersive equations whose well-posedness theory for the Cauchy problem on the whole plane has been extensively studied; see, for instance, \cite{Faminskii-95, Kinoshita-21, LinPas-09, MolPil-15}.  In the periodic setting, the pioneering work \cite{LinPanRo-19} established local well-posedness of the ZK equation in the Sobolev spaces $H^s(\mathbb{T}\times\mathbb{T})$ for $s>5/3$, relying on a variant of Strichartz-type estimates. This regularity threshold was later improved to $s>1$ in \cite{KinSch-21} using a nonlinear Loomis--Whitney inequality. More recently, Osawa \cite{Osawa-22} proved local well-posedness for the two-dimensional ZK equation on the mixed domain $\mathbb{R}\times\mathbb{T}$, for $9/10 < s < 1$.

There has been significant progress in control theory aimed at understanding how damping mechanisms influence the energy of systems governed by partial differential equations. In particular, the exponential stability of dispersive equations related to water waves on bounded domains has been extensively investigated. For instance, it is well known that the Korteweg-de Vries (KdV) equation~\cite{Zuazua2002}, the Boussinesq system of KdV--KdV type~\cite{Pazoto2008}, the Kawahara equation~\cite{Araruna2012}, among others, are exponentially stable by means of the Compactness--Uniqueness method introduced by J.~L.~Lions~\cite{Lions1988}. Other stabilization results, such as those obtained in~\cite{Cerpa2021} and~\cite{Capistrano2018}, rely on Urquiza's method and the Backstepping approach. In all these works, the stabilization is achieved through damping mechanisms acting either inside the equation or through the boundary.

Recent works \cite{martinez2022,boumediene} established exponential decay for a fifth-order KdV-type equation using the Compactness–Uniqueness method combined with Lyapunov techniques. Exponential stabilization results for KP-II and K-KP-II were also obtained in \cite{Panthee2011, ailton2021}. Moreover, \cite{CF-K-P} proved internal observability and rapid stabilization for the linear model associated with \eqref{ZKeq}, while \cite{Pignotti} addressed a KdV–Burgers equation with delay and damping on the real line.

%More recently, in~\cite{martinez2022,boumediene}, the authors established exponential decay for a fifth-order KdV-type equation using the Compactness--Uniqueness argument combined with Lyapunov techniques. In addition, exponential decay for the KP-II and K-KP-II equations was obtained in~\cite{Panthee2011} and~\cite{ailton2021}\footnote{See also the references therein for further stabilization results concerning KP-II and K-KP-II.}. In both works, the authors proved well-posedness and regularity of the corresponding models and showed that the associated energy decays exponentially when a localized damping term is present. We also mention a recent result concerning the linear equation associated with \eqref{ZKeq} posed on a periodic domain. In \cite{CF-K-P}, using tools from nonharmonic Fourier series, the authors established several internal observability theorems. As a consequence, by applying a duality principle together with a general feedback construction, they obtained various exact controllability results as well as rapid uniform stabilization.  Also, we would mention \cite{Pignotti} which explores the generalized Korteweg-de Vries-Burgers equation incorporating delay feedback and a damping term in the real line.

We also mention that Chen \cite{Chen} proved a boundary Carleman estimate for the ZK equation on bounded domains, yielding a unique continuation property. Later, Chen and Rosier \cite{Chen3} used the flatness approach to obtain null controllability for the linear ZK equation on rectangles with Dirichlet boundary control. Chen also established global approximate controllability for the KdV equation on the torus \cite{Chen1} and, more recently, extended the Agrachev–Sarychev method to the ZK–Burgers equation \cite{Chen2}, proving approximate controllability and irreducibility of the associated stochastic semigroup.

%We point out that Chen \cite{Chen} established a boundary Carleman estimate for the ZK equation posed on a bounded domain, thereby obtaining a unique continuation property for this model. Subsequently, Chen and Rosier \cite{Chen3} studied the linear Zakharov–Kuznetsov equation on a rectangular domain with a left Dirichlet boundary control and, using the flatness approach, proved its null controllability and characterized a space of analytically reachable states. In another development, Chen \cite{Chen1} investigated the KdV equation on the periodic domain and showed that it is globally approximately controllable by a two-dimensional external force, relying on the Agrachev–Sarychev method from geometric control theory. More recently, Chen \cite{Chen2} extended these ideas to the two-dimensional ZK–Burgers equation on the torus, proving its global approximate controllability by a finite-dimensional external force via the same Agrachev–Sarychev approach, and, as an application, establishing the irreducibility of the transition semigroup associated with the corresponding stochastic equation driven by degenerate noise.

The stabilization of multidimensional ZK models on bounded domains has also been significantly advanced by the works of Doronin and Larkin. In \cite{DoLa}, they study the equation on rectangles and strips, establishing spectral properties, identifying critical domain sizes, and proving exponential decay. In \cite{DoLa1}, they provide boundary conditions that remove these size restrictions and extend the analysis to 2D strips and 3D channels. Finally, Larkin \cite{La} addresses the fully three-dimensional case, proving global existence and exponential decay for sufficiently small initial data.
%The study of stabilizing multidimensional ZK models on bounded domains has been significantly advanced by a series of works of Doronin and Larkin. Let us comment on three of them. In \cite{DoLa}, the authors analyze the initial–boundary value problems for the two-dimensional ZK equation posed on bounded rectangles and on a strip, establishing spectral properties of the linearized operator, identifying critical domain sizes, and proving exponential decay of regular solutions for the nonlinear system. In a complementary study \cite{DoLa1}, they investigate the linear two-dimensional ZK equation on rectangles and channel-like strips, providing boundary conditions that eliminate critical size restrictions and extending the analysis to 2D strips and 3D grooves, again obtaining exponential decay of solutions. Finally, Larkin \cite{La} considers the three-dimensional ZK equation on bounded domains, proving global existence and uniqueness of regular solutions and showing exponential decay of the $H^{2}$-norm for sufficiently small initial data, thereby extending these stabilization results to the fully three-dimensional setting.

\subsection{Problem setting}
These contributions underscore the mathematical interest in analyzing the asymptotic behavior of solutions to equations of the form \eqref{ZKeq}.  Inspired by the works~\cite{martinez2022,Chen,boumediene,ailton2021,Panthee2011,Valein}, we study the qualitative properties of the initial--boundary value problem for the Zakharov--Kuznetsov equation posed on the bounded domain $\Omega = (0,L)\times(0,L) \subset \mathbb{R}^{2}$, incorporating localized damping and delay terms, namely:
\begin{equation}\label{eq:ZK}
\begin{cases}
	\begin{aligned} 
		&\partial_t\zeta(x,y,t) + \partial_x\left(\alpha \partial^2_x\zeta(x,y,t) + \gamma\partial^2_y \zeta(x,y,t) + \frac{1}{2}\zeta^2(x,y,t)\right)\\
		&+ a(x,y)\zeta(x,y,t) + b(x,y)\zeta(x,y,t-h)=0,
	\end{aligned} & (x,y)\in \Omega,\ t>0. \\
	\zeta(0,y,t)=\zeta(L,y,t)=\partial_x\zeta(L,y,t) = \partial_y\zeta(L,y,t)=0,& y\in(0,L),\ t\in(0,T), \\
	\zeta(x,L,t)=\zeta(x,0,t)=0,& x\in(0,L),\ t\in(0,T),\\
	\zeta(x,y,0)=\zeta_0(x,y), \quad
	\zeta(x,y,t)= z_0(x,y,t), & (x,y)\in\Omega,\ t\in(-h,0).
\end{cases}
\end{equation}
Here $h>0$ is the time delay,   $\alpha>0$, $\gamma>0$  are real constants.  For our purpose,  let us consider the following assumption.

	\begin{assumption}\label{A1}
		The real functions $a\left(x,y\right)$ and $b\left(x,y\right)$ are nonnegative belonging to $L^\infty(\Omega)$. Moreover, $a(x,y) \geq a_0>0$ is almost everywhere in a nonempty open subset $\omega \subset \Omega$.
	\end{assumption}

Our objective here is to introduce, for the first time, the ZK system incorporating not only a localized damping term $a(x,y)u$, which acts as a feedback--dissipation mechanism (see, for instance, \cite{ailton2021,Roger2025}), but also an \emph{anti-damping} component, known as delay term wich change completly the behavior of the energy. 

In this framework, as mentioned, the feedback term may counterbalance the dissipative effects, so the system no longer ensures a monotonic decay of the total energy. Thus, our goal is to demonstrate that the energy associated with the solutions of system~\eqref{eq:ZK}
\begin{equation}\label{eq:ZKen}
\begin{split}
\mathcal{E}_{\zeta}(t) =&\frac{1}{2} \int_0^L\int_0^L \zeta^2(x,y,t) \,dx\,dy +\frac{h}{2} \int_0^L\int_0^L\int_0^1 b(x,y)\zeta^2(x,y,t-\rho h) \,d\rho\,dx\,dy.
\end{split}
\end{equation}
decays exponentially under appropriate conditions. Precisely, we want to answer the following question:

\vspace{0.2cm}

\textit{Does $ \mathcal{E}_\zeta(t) \rightarrow 0$ as $ t \rightarrow \infty$? If it is the case, can we give the decay rate?}

\subsection{Notations and main results} Before addressing this question, we recall the functional setting required for our analysis.  
Let $\Omega \subset \mathbb{R}^{2}$ be an open domain and let $k \in \mathbb{Z}^{+}$.  
We denote by $H^{k}(\Omega)$ the classical Sobolev space defined by
$$
	H^{k}(\Omega)
	\coloneqq
	\left\{
	f \in L^{2}(\Omega)\;:\;
	\partial_x^{\alpha} f \in L^{2}(\Omega)\ \text{(in the distributional sense)},\ 
	|\alpha| \le k
	\right\},
$$
equipped with the norm
\[
\|f\|_{H^{k}(\Omega)}
=
\left(
\sum_{|\alpha|\le k}
\int_{\Omega}
\bigl|\partial_x^{\alpha} f(x)\bigr|^{2}\,dx
\right)^{1/2}.
\]
We also define $H_{0}^{k}(\Omega)$ as the closure of $C_{0}^{\infty}(\Omega)$ in $H^{k}(\Omega)$. 

Before presenting the first result, let us define some constants. Define $T_{0}$ by
\begin{equation}\label{time}
	T_{0}
	=
	\frac{1}{2\theta}
	\ln\!\left( \frac{2\xi \kappa}{\mu} \right)
	+1,
\end{equation}
where $L>0$, $\xi>1$, $0<\mu<1$,
\[
\theta
=
\frac{3\alpha\eta}{(1+2\eta L)L^{2}},
\qquad
\kappa
=
1+\max\left\{ 2\eta L,\ \frac{\sigma}{\xi} \right\},
\qquad \text{and}\qquad \eta \in \left(0, \frac{\xi-1}{2L(1+2\xi)}\right).\] 
Here, these constants are chosen to satisfy
\[
\frac{2\alpha \eta}{(2+2\eta L)L^{2}}
=
\frac{\sigma}{2h(\xi+\sigma)},
\qquad \text{and}\qquad
\sigma = \xi - 1 - 2L\eta(1+2\xi).
\]
Define also the quantity $T_{\min}>0$ by
\begin{equation}\label{tmin}
T_{\min}
:=
-\frac{1}{\nu}\ln\!\left(\frac{\mu}{2}\right)
+
\left(
	\frac{2\|b\|_{\infty}}{\nu}+1
\right)T_{0},
\end{equation}
with
\[
\nu
=
\frac{1}{T_{0}}
\ln\!\left( \frac{1}{\mu+\varepsilon} \right)
\]
and $T_{0}$ is given by \eqref{time}.

The first result of this manuscript shows that, without imposing any restrictive condition on the length $L$ of the domain and assuming that the weight of the delayed feedback is sufficiently small, the energy \eqref{eq:ZKen} associated with the solution of system \eqref{eq:ZK} is locally stable.
\begin{theorem}[Optimal local stabilization]\label{th:ZKes}
Assume that the functions $a(x,y)$ and $b(x,y)$ satisfy the conditions stated in Assumption~\ref{A1}.  
Then there exist constants $\delta>0$, $r>0$, $C>0$, and $\gamma>0$, depending on  
$T_{\min}$, $\xi$, $L$, and $h$, such that if $\|b\|_{\infty} \leq \delta$, such that for every  
$(\zeta_{0}, z_{0}) \in \mathcal{H} = L^{2}(\Omega) \times L^{2}(\Omega \times (0,1))$  
satisfying $\|(\zeta_{0}, z_{0})\|_{\mathcal{H}} \leq r$,  
the energy of system~\eqref{eq:ZK} satisfies
\[
\mathcal{E}_{\zeta}(t)
\leq
C e^{-\gamma t}\, \mathcal{E}_{\zeta}(0),
\qquad \forall\, t > T_{\min}.
\]
\end{theorem}

Following the ideas in \cite{martinez2022}, we now derive stability properties for the system introduced below, which we refer to as the $\mu_i$--system. Observe that, by choosing in~\eqref{eq:ZK} the coefficients  $a(x,y)=\mu_{1}a(x,y)$ and $b(x,y)=\mu_{2}a(x,y)$, with $\mu_{1}$ and $\mu_{2}$ real constants, we obtain the system
\begin{equation}\label{eq:MU}
\begin{cases}
	\begin{aligned} 
	&	\partial_t\zeta(x,y,t) + \partial_x\left(\alpha \partial_{x}^2\zeta(x,y,t) +  \gamma\partial^2_y \zeta(x,y,t)	+ \frac{1}{2}\zeta^2(x,y,t)\right)\\& + a(x,y)\left(\mu_1\zeta(x,y,t) + \mu_2 \zeta(x,y,t-h)\right)=0,
	\end{aligned}& (x,y,t)\in \Omega\times\mathbb{R}^+\\
	\zeta(0,y,t)=\zeta(L,y,t)=\partial_x \zeta(L,y,t) = \partial_y\zeta(L,y,t)=0,&y\in(0,L), \ t\in(0,T), \\
	\zeta(x,L,t)=\zeta(x,0,t)=0,& x\in(0,L), \ t\in(0,T),\\
	\zeta(x,y,0)=\zeta_0(x,y),\quad
	\zeta(x,y,t)= z_0(x,y,t), & (x,y)\in \Omega, \ t\in(-h,0).
\end{cases}
\end{equation}
Here, $\mu_{1} > \mu_{2}$ are positive real numbers, and $a(x,y)$ satisfies Assumption~\ref{A1}. We define the total energy associated with~\eqref{eq:MU} by
\begin{equation}\label{eq:MUen}
\begin{split}
\mathcal{E}_{\zeta}(t)
=&\frac{1}{2}
\int_{0}^{L}\int_{0}^{L} \zeta^{2}(x,y,t)\,dx\,dy
+\frac{\xi}{2}
\int_{0}^{L}\int_{0}^{L}\int_{0}^{1}
a(x,y)\,\zeta^{2}(x,y,t-\rho h)\,d\rho\,dx\,dy,
\end{split}
\end{equation}
where $\xi>0$ satisfies
\begin{equation}\label{eq:MUcond}
    h\mu_{2} < \xi < h(2\mu_{1}-\mu_{2}).
\end{equation}

Note that the derivative of the energy \eqref{eq:MUen} satisfies
\begin{equation*}
\begin{aligned}
\frac{d}{dt}\mathcal{E}_{\zeta}(t)
\leq -C
\Bigg(
&
	\int_{0}^{L} \left(\partial_{x}\zeta(0,y,t)\right)^{2}dy
	+ \int_{0}^{L} \left(\partial_{y}\zeta(x,0,t)\right)^{2}dx
\\
&
	+ \int_{0}^{L}\int_{0}^{L} a(x,y)\,\zeta^{2}(x,y,t-h)\,dx\,dy
\Bigg),
\end{aligned}
\end{equation*}
for some constant $C = C(\alpha,\gamma,\mu_{1},\mu_{2},\xi,h)\ge 0$.  
Thus, the function \(a(x,y)\) acts as a feedback-damping mechanism for the linearized problem.

To establish our second and third results, we introduce the following quantities.  
Let $\theta>0$ be chosen so that
\begin{equation}\label{TTheta}
\theta
<
\min\left\{
	\frac{\eta}{(1+2\eta L)L^{2}}
	\left[
		3\alpha 
		- \frac{1}{2} C^{4/3} r^{4/3} L^{10/3}
	\right],
	\ \frac{\xi\sigma}{2h(\xi+\sigma\xi)}
\right\},
\end{equation}
and where the positive constants $\eta$ and $\sigma$ satisfy
\[
\sigma < \frac{2h}{\xi}
\left(
	\mu_{1} - \frac{\mu_{2}}{2} - \frac{\xi}{2h}
\right),
\]
and
\[
\eta <
\min\left\{
	\frac{1}{2L\mu_{2}}
	\left( \frac{\xi}{h} - \mu_{2} \right),
	\,
	\frac{1}{2L\mu_{1} + L\mu_{2}}
	\left(
		\mu_{1} - \frac{\mu_{2}}{2} - \frac{\xi}{2h}(1+\sigma)
	\right)
\right\}.
\]

For system~\eqref{eq:MU}, the analysis is divided into two steps.  
Using an appropriate Lyapunov functional, we first prove that the energy $\mathcal{E}_{\zeta}(t)$ decays exponentially as $t\to\infty$, provided the initial data are sufficiently small.  
This gives our second local stabilization result:

\begin{theorem}[Local stabilization]\label{th:MUlyes}
Let $L>0$. Assume that $a(x,y)\in L^{\infty}(\Omega)$ is non-negative and that condition~\eqref{eq:MUcond} holds.  
Then, there exists
\begin{equation}\label{RRRRR}
    0 < r < \frac{\sqrt[4]{216\alpha^{3}}}{C\,L^{5/2}}
\end{equation}
such that, for every  
\(
\left(\zeta_{0}, z_{0}(\cdot,\cdot,-h(\cdot)))\right)\in\mathcal{H}
\)
satisfying  
\(
\|(\zeta_{0},z_{0}(\cdot,\cdot,-h(\cdot)))\|_{\mathcal{H}} \le r,
\)
the energy defined in~\eqref{eq:MUen} decays exponentially.  
More precisely, there exist constants \(\theta>0\) and \(\kappa>0\) such that
\[
\mathcal{E}_{\zeta}(t) \le \kappa\,\mathcal{E}_{\zeta}(0)\,e^{-2\theta t}, 
\qquad \forall t>0,
\]
where $\kappa = 1+\max\{2\eta L,\sigma\}$ and $\theta$ satisfies \eqref{TTheta}.
\end{theorem}

The last result of this work concerns the global behavior of system~\eqref{eq:MU}.   Here, the restriction on the smallness of the initial data is removed.  
To achieve this, we employ the compactness--uniqueness method of J.-L.~Lions~\cite{Lions,Lions1988},   reducing the problem to establishing an \emph{observability inequality} for the nonlinear system.

\begin{theorem}[Global stabilization]\label{th:MUes}
Let $a(x,y)\in L^{\infty}(\Omega)$ satisfy Assumption~\ref{A1}.  
Suppose that $\mu_{1} > \mu_{2}$ satisfies condition~\eqref{eq:MUcond}.  
Then, for every $R>0$, there exist constants $C=C(R)>0$ and $\nu=\nu(R)>0$ such that  
the energy \(\mathcal{E}_{\zeta}(t)\), defined in~\eqref{eq:MUen}, decays exponentially as \(t\to\infty\),  
whenever \(\|(\zeta_{0},z_{0})\|_{\mathcal{H}}\le R\).
\end{theorem}

\subsection{Outline of the article} We conclude the introduction by outlining the structure of the manuscript. 
The paper is organized as follows. In Section~\ref{Sec1}, we revisit the well-posedness theory for the systems considered in this work. Section~\ref{Sec2} is devoted to establishing our first, and optimal, local stability result, namely Theorem~\ref{th:ZKes}. In Section~\ref{Sec3}, we prove the exponential stability stated in Theorem~\ref{th:MUlyes} for the energy associated with the $\mu_i$--system~\eqref{eq:MU}. In the same section, we further extend this local stability to a global result by proving Theorem~\ref{th:MUes}. Finally, further comments are presented in Section~\ref{Sec4}.

\section{Well-Posedness: An overview}\label{Sec1}In this section, we examine the $\mu_i$--system~\eqref{eq:MU}, which plays a fundamental role in obtaining the stabilization results for~\eqref{eq:ZK}. Since the arguments involved are classical, we present only the main statements and outline the essential ideas of the proofs.

\subsection{Linear system} In this section, we use semigroup theory to obtain well-posedness results for the linear ZK system associated with~\eqref{eq:MU}.  
To this end, consider the auxiliary function
\[
z(x,y,\rho,t)=\zeta(x,y,t-\rho h), \quad (x,y)\in\Omega,\ \rho\in(0,1),\ t>0,
\]
which satisfies the transport equation
\begin{equation}\label{eq:ZKtp}
	\begin{cases}
		h\partial_t z(x,y,\rho,t)+\partial_{\rho}z(x,y,\rho,t)=0, & (x,y)\in\Omega,\ \rho\in(0,1),\ t>0,\\[0.3em]
		z(x,y,0,t)= \zeta(x,y,t), & (x,y)\in\Omega,\ t>0,\\[0.3em]
		z(x,y,\rho,0)=z_0(x,y,\rho,-\rho h), & (x,y)\in\Omega,\ \rho\in(0,1).
	\end{cases}
\end{equation}

Consider $\mathcal{H}=L^2(\Omega)\times L^2(\Omega\times(0,1))$ be the Hilbert space equipped with the inner product
\begin{equation*}
	\begin{split}
		\langle ( \zeta,z),(v,w)\rangle_{\mathcal{H}}
		=&\int_\Omega  \zeta(x,y)v(x,y)\,dx\,dy+\xi\|a\|_\infty\int_\Omega\int_0^1 z(x,y,\rho)w(x,y,\rho)\,d\rho\,dx\,dy,
	\end{split}
\end{equation*}
where $\xi$ satisfies~\eqref{eq:MUcond}. To study well-posedness in the Hadamard sense, we rewrite the linear system as an abstract Cauchy problem.  Let $U(t)=(\zeta(\cdot,\cdot,t), z(\cdot,\cdot,\cdot,t))$ and denote $z(1):=z(x,y,1,t)$.  From the linearized ZK system and~\eqref{eq:ZKtp}, we obtain
\begin{equation}\label{eq:ZKlin1}
	\left\{
	\begin{aligned}
		&\partial_t \zeta(x,y,t)
		+ \alpha\,\partial_x^3 \zeta(x,y,t)
		+ \gamma\,\partial_x\partial_y^2 \zeta(x,y,t)
		+ a(x,y)\big(\mu_1  \zeta(x,y,t) + \mu_2 z(1)\big) = 0,
		\\[0.4em]
		&\zeta(0,y,t)=\zeta(L,y,t)=\partial_x \zeta(L,y,t)=\partial_y\zeta(L,y,t)=0,
		\quad y\in(0,L),
		\\[0.4em]
		&\zeta(x,0,t)=\zeta(x,L,t)=0,
		\quad x\in(0,L),
		\\[0.4em]
		&\zeta(x,y,0)=\zeta_0(x,y),
		\quad (x,y)\in\Omega,
		\\[0.4em]
		&h\,\partial_t z(x,y,\rho,t) + \partial_{\rho} z(x,y,\rho,t) = 0,
		\quad (x,y)\in\Omega,\ \rho\in(0,1),\ t>0,
		\\[0.4em]
		&z(x,y,0,t)=\zeta(x,y,t),
		\quad (x,y)\in\Omega,\ t>0,
		\\[0.4em]
		&z(x,y,\rho,0)=z_0(x,y,\rho,-\rho h),
		\quad (x,y)\in\Omega,\ \rho\in(0,1).
	\end{aligned}
	\right.
\end{equation}
This system can be written abstractly as
\begin{equation}\label{eq:ZKabs}
	\begin{cases}
		\dfrac{d}{dt}U(t)=AU(t),\\[0.3em]
		U(0)=(\zeta_0,z_0(x,y,-\rho h)),
	\end{cases}
\end{equation}
where $A\colon D(A)\subset\mathcal{H}\to\mathcal{H}$ is defined by
\begin{equation}\label{eq:ZKA}
	A(\zeta,z)=\left(
	-\alpha\,\partial_x^3\zeta - \gamma\,\partial_x\partial_y^2\zeta - a(x,y)(\mu_1\zeta+\mu_2z(1)),
	-\,h^{-1}\partial_{\rho}z
	\right),
\end{equation}
with dense domain
\begin{equation*}
	D(A)=
	\left\{
	\begin{aligned}
		&(\zeta,z)\in\mathcal{H}:\ \zeta\in H^3(\Omega),\ \partial_{\rho}z\in L^2(\Omega\times(0,1)),\\
		&\zeta(0,y)=\zeta(L,y)=\zeta(x,0)=\zeta(x,L)=0,\\
		&\partial_x\zeta(L,y)=\partial_y\zeta(L,y)=0,\\
		&z(x,y,0)=\zeta(x,y)
	\end{aligned}
	\right\}.
\end{equation*}
Taking into account these definitions, we have the following results.
\begin{lemma}
	The operator $A$ is closed, and its adjoint $A^\ast\colon D(A^\ast)\subset\mathcal{H}\to\mathcal{H}$ is given by
	\begin{equation*}
		A^\ast(\zeta,z)=\left(
		\alpha\,\partial_x^3\zeta+\gamma\,\partial_x\partial_y^2\zeta - a(x,y)\mu_1\zeta+\frac{\xi\|a\|_\infty}{h}z(\cdot,\cdot,0),
		\,h^{-1}\partial_{\rho}z
		\right),
	\end{equation*}
	with dense domain
	\begin{equation*}
		D(A^\ast)=
		\left\{
		\begin{aligned}
			&(\zeta,z)\in\mathcal{H}:\ \zeta\in H^3(\Omega),\ \partial_{\rho}z\in L^2(\Omega\times(0,1)),\\
			&\zeta(0,y)=\zeta(L,y)=\zeta(x,0)=\zeta(x,L)=0,\\
			&\partial_x\zeta(L,y)=\partial_y\zeta(L,y)=0,\\
			&z(x,y,1)=-\dfrac{a(x,y)h\mu_2}{\xi\|a\|_\infty}\zeta(x,y)
		\end{aligned}
		\right\}.
	\end{equation*}
\end{lemma}

Moreover, the operator $A$ generates a semigroup. 
\begin{proposition}\label{pr:ZKdis}
	Assume that $a\in L^\infty(\Omega)$ is nonnegative and that~\eqref{eq:MUcond} holds.  
	Then $A$ is the infinitesimal generator of a $C_0$-semigroup on $\mathcal{H}$.
\end{proposition}

\begin{proof}
	Let $U=(\zeta,z)\in D(A)$, where $A$ is given by \eqref{eq:ZKA}. 
	Multiplying by $U$ in $\mathcal{H}$ and integrating by parts, we obtain
$$
		\langle AU,U\rangle_{\mathcal{H}}\le \frac{\xi\|a\|_\infty}{2h}\int_\Omega  \zeta^2(x,y)\,dx\,dy.
$$
	Hence, for $\lambda=\frac{\xi\|a\|_\infty}{2h}$ we have $\langle(A-\lambda I)U,U\rangle_{\mathcal{H}}\le0$.  
	Since $A-\lambda I$ is densely defined, closed, and both $A-\lambda I$ and $(A-\lambda I)^\ast$ are dissipative, the Lumer–Phillips theorem ensures that $A$ generates a $C_0$-semigroup on $\mathcal{H}$.
\end{proof}

The next theorem is a direct consequence of semigroup theory. 
\begin{theorem}
	Assume that $a\in L^\infty(\Omega)$ and~\eqref{eq:MUcond} holds.  
	Then, for each initial data $U_0\in\mathcal{H}$ there exists a unique mild solution $U\in C([0,\infty),\mathcal{H})$ for system~\eqref{eq:ZKabs}.  
	Moreover, if $U_0\in D(A)$, the solution is classical $U\in C([0,\infty),D(A))\cap C^1([0,\infty),\mathcal{H}).$
\end{theorem}

Now, let us present the following derivative formula of the energy. 
\begin{proposition}
	Let $a\in L^\infty(\Omega)$ be nonnegative and assume that~\eqref{eq:MUcond} holds.  
	Then, for any mild solution of~\eqref{eq:ZKabs}, the energy $\mathcal{E}_\zeta$, defined by~\eqref{eq:MUen}, is non-increasing and there exists $C>0$ such that
	\begin{equation}\label{eq:ZKendis}
		\begin{split}
			\frac{d}{dt}\mathcal{E}_\zeta(t)
			\le -C\Big(
			&\int_0^L (\partial_x\zeta(0,y,t))^2\,dy
			+\int_0^L (\partial_y\zeta(0,y,t))^2\,dy\\
			&+\int_\Omega a(x,y)\zeta^2\,dx\,dy
			+\int_\Omega a(x,y)\zeta^2(x,y,t-h)\,dx\,dy
			\Big),
		\end{split}
	\end{equation}
	where
	\[
	C=\min\left\{
	\frac{\alpha}{2},\frac{\gamma}{2},
	\mu_1-\frac{\mu_2}{2}-\frac{\xi}{2h},
	-\frac{\mu_2}{h}+\frac{\xi}{2h}
	\right\}.
	\]
\end{proposition}

\begin{proof}
	Multiply the first equation of \eqref{eq:ZKlin1} by $\zeta(x,y,t)$ and integrate by parts over $\Omega$.  
	Then multiply the fifth equation of \eqref{eq:ZKlin1} by $z(x,y,\rho,t)$ and integrate over $\Omega\times(0,1)$.  
	Adding the two identities and using boundary conditions yields the desired inequality.
\end{proof}

To derive Kato-type smoothing and contraction estimates, for $T>0$, define the spaces
\begin{equation*}
	\begin{split}
		\mathcal{B}_X &= C([0,T],L^2(\Omega))\cap L^2(0,T;H^1_{0}(\Omega)),\\
	%	\mathcal{B}_H &= C([0,T],L^2(\Omega))\cap L^2(0,T;H^2(\Omega)),
	\end{split}
\end{equation*}
with natural norms
\begin{equation*}
	\begin{aligned}
		\|y\|_{\mathcal{B}_X}
		&=\max_{t\in[0,T]}\|y(\cdot,\cdot,t)\|_{L^2(\Omega)}
		+\left(\int_0^T\|y(\cdot,\cdot,t)\|_{H^1_{0}(\Omega)}^2\,dt\right)^{1/2}.%\\[0.3em]
	%	\|y\|_{\mathcal{B}_H}
	%	&=\max_{t\in[0,T]}\|y(\cdot,\cdot,t)\|_{L^2(\Omega)}
	%	+\left(\int_0^T\|y(\cdot,\cdot,t)\|_{H^2(\Omega)}^2\,dt\right)^{1/2}.
	\end{aligned}
\end{equation*}
With this in hand, the following result is given. 
\begin{proposition}\label{pr:ZKreg}
	Let $a\in L^\infty(\Omega)$ be nonnegative.  
	Then the mapping
	\[
	(\zeta_0,z_0(\cdot,\cdot,-h(\cdot)))\in\mathcal{H}
	\mapsto
	(\zeta,z)\in\mathcal{B}_X\times C([0,T],L^2(\Omega\times(0,1)))
	\]
	is continuous, and the following estimates hold:
	\begin{equation}\label{eq:ZKr1}
		\begin{split}
			&\frac{1}{2}\int_\Omega \zeta^2(x,y)\,dx\,dy
			+\frac{\xi}{2}\int_\Omega\int_0^1 a(x,y)\zeta^2(x,y,t-\rho h)\,d\rho\,dx\,dy\\
			&\le \frac{1}{2}\int_\Omega \zeta_0^2\,dx\,dy
			+\frac{\xi}{2}\int_\Omega\int_0^1 a(x,y)z_0^2(x,y,-\rho h)\,d\rho\,dx\,dy,
		\end{split}
	\end{equation}
	\begin{equation}\label{eq:ZKr2}
		\frac{3\alpha}{2}\int_0^T\!\!\int_\Omega (\partial_x\zeta)^2\,dx\,dy\,dt
		\le C(a,\mu_1,\mu_2,L)(1+T)
		\|(\zeta_0,z_0(\cdot,\cdot,-h(\cdot)))\|_{\mathcal{H}}^2,
	\end{equation}
	and
	\begin{equation}\label{eq:ZKr3}
		\begin{split}
			\|\zeta_0\|_{L^2(\Omega)}^2
			\le \frac{1}{T}\int_0^T\!\!\int_\Omega \zeta^2(x,y,t)\,dx\,dy\,dt
			+\gamma\int_0^T\!\!\int_0^L(\partial_y\zeta(0,y,t))^2\,dy\,dt\\
			+(\mu_1+\mu_2)\int_0^T\!\!\int_\Omega a(x,y)\zeta^2(x,y,t)\,dx\,dy\,dt
			+\mu_2\int_0^T\!\!\int_\Omega a(x,y)\zeta^2(x,y,t-h)\,dx\,dy\,dt.
		\end{split}
	\end{equation}
\end{proposition}

\begin{proof}
	The argument follows from multiplier techniques (Morawetz-type identities).  
	Estimate~\eqref{eq:ZKr1} follows from~\eqref{eq:ZKendis}.  
	For~\eqref{eq:ZKr2}, multiply the fifth equation of \eqref{eq:ZKlin1} by $z(x,y,\rho,t)$ and the first one by $x\,\zeta(x,y,t)$, integrate by parts in $\Omega\times(0,T)$, and combine the results.  Finally, multiply the first equation of \eqref{eq:ZKlin1} by $(T-t)\zeta(x,y,t)$ and integrate over $\Omega\times(0,T)$ gives~\eqref{eq:ZKr3}.
\end{proof}

\subsection{Linear system with source term}

We will study the system \eqref{eq:ZKlin1}, with a source term $f(x,y,t)$ on the right-hand side. The next result ensures the well-posedness of this system.

\begin{proposition}
	Assume that  $a(x,y)\in L^\infty(\Omega)$ is a nonnegative function and that \eqref{eq:MUcond} is satisfied. For any $\left(\zeta_0,z_0(\cdot,\cdot, -h(\cdot))\right)\in \mathcal{H}$ and $f\in L^1\left(0,T,L^2(\Omega)\right)$, there exists a unique mild solution for \eqref{eq:ZKlin1} with the source term $f(x,y,t)$ on the right-hand side in the class
	$$
	\left(\zeta,\zeta(\cdot,\cdot, t-h(\cdot))\right)\in \mathcal{B}_X\times C\left([0,T], L^2(\Omega\times(0,1))\right).
	$$
	Moreover, we have
	\begin{equation}\label{eq:ZKsou1}
		\left\lVert{(\zeta,z)}\right\rVert_{C([0,T],\mathcal{H})}\leq e^{\frac{\xi\left\lVert{a}\right\rVert_\infty}{2h}T}
		\left(% 
		\left\lVert{(\zeta_0,z_0(\cdot,\cdot,-h(\cdot)))}\right\rVert_{\mathcal{H}} 
		+\left\lVert{f}\right\rVert_{L^1(0,T,L^2(\Omega))}
		\right)
	\end{equation}
	and
	\begin{equation}\label{eq:ZKsou2}
		\delta\left\lVert{\zeta}\right\rVert_{L^2(0,T,H^1(\Omega))}^2
		\leq \mathcal{C}
		\left(% 
		\left\lVert (\zeta_0,z_0 (\cdot,\cdot,-h(\cdot)))\right\rVert_{\mathcal{H}}^2
		+\left\lVert{f}\right\rVert_{L^1(0,T,L^2(\Omega))}^2
		\right)
	\end{equation}
	where  
	\begin{equation*}
		\mathcal{C} = \mathcal{C}
		\left(a,\mu_1,\mu_2,L,T,h\right)=\frac{3L}{2}+L\left\lVert{a}\right\rVert_\infty(\mu_1+\mu_2)+\delta\left(1+T+e^{\frac{\xi\left\lVert{a}\right\rVert_\infty}{h}T}\right)
	\end{equation*}
	and $\delta=\min\left\lbrace{1,\frac{3\alpha}{2}}\right\rbrace$.
\end{proposition}

\begin{proof}
	Note that $A$ is an infinitesimal generator of a $C_0$-semigroup $(e^{tA})_{t\geq0}$ satisfying $$\left\lVert{e^{tA}}\right\rVert_{\mathcal{L}(\mathcal{H})}\leq e^{\frac{\xi\left\lVert{a}\right\rVert_\infty}{2h}t}$$ and the system can be rewritten as a first order system with source term $(f(\cdot,\cdot,t),0)$, showing the well-posedness in $C([0,T],\mathcal{H})$.  Finally, observe that the right-hand side is not homogeneous, since
	\begin{equation*}
		\left\lvert%
		\int_0^T\int_{\Omega} xf(x,y,t) \zeta(x,y,t)\,dx\,dy\,dt
		\right\rvert%
		\leq \frac{L}{2} \left\lVert{\zeta}\right\rVert_{C([0,T]; L^2(\Omega))}^2+ \frac{L}{2}\left\lVert{f}\right\rVert_{L^1(0,T, L^2(\Omega))}^2,
	\end{equation*}
	showing the result.
\end{proof}

\subsection{Nonlinear system: Global results.} In this last part, we consider the nonlinear term $uu_x$ as a source term.

\begin{proposition}\label{pr:ZKcont}
	If $\zeta\in \mathcal{B}_X$ then $\zeta\partial_x\zeta\in L^1(0,T; L^2(\Omega))$ and the map
	$$
	\zeta\in\mathcal{B}_X \mapsto \zeta\partial_x\zeta\in L^1(0,T; L^2(\Omega))
	$$
	is continuous. In particular, there exists $K>0$,  such that, for all $\zeta,v\in \mathcal{B}_X$ we have
	\begin{equation}\label{eq:ZKnl}
		\left\lVert{\zeta\partial_x\zeta-v\partial_xv}\right\rVert_{L^1(0,T,L^2(\Omega))}
		\leq K
		\left(% 
		\left\lVert{\zeta}\right\rVert_{\mathcal{B}_X}
		+\left\lVert{v}\right\rVert_{\mathcal{B}_X}
		\right)\left\lVert{\zeta-v}\right\rVert_{\mathcal{B}_X}.
	\end{equation}
\end{proposition}

\begin{proof}
	The Hölder inequality and the Sobolev embedding $H_0^1(\Omega) \hookrightarrow L^4(\Omega)$ gives us \eqref{eq:ZKnl} with $K=CT^{\frac{1}{4}}$,
%\begin{equation}\label{eq:ZKnl1}
%		\left\lVert \zeta\partial_x\zeta-v\partial_xv \right\rVert_{L^1(0,T, L^2(\Omega))}
%		\leq C_1 C T^{\frac{1}{4}}\left(\left\lVert{\zeta}\right\rVert_{\mathcal{B}_{H}} + \left\lVert{v}\right\rVert_{\mathcal{B}_{H}}
%		\right)\left\lVert{\zeta-v}\right\rVert_{\mathcal{B}_H},
%	\end{equation}
for $u,v\in \mathcal{B}_X$.  
%Note that $u\in\mathcal{B}_X$ implies 
%	$\zeta(\cdot,\cdot, t)\in H_{0}^2(\Omega)$
	%and consequently 
%	$\zeta(\cdot,\cdot,t)\in H_{0}^1(\Omega)$ and
%	$\partial_x\zeta(\cdot,\cdot,t)\in H_{0}^1(\Omega)$. 
This proves the continuity of the nonlinear map.  So,  from~\eqref{eq:ZKnl}, with $v=0$, we conclude that $\zeta\partial_x\zeta \in L^1(0,T, L^2(\Omega))$, and the proof is complete.
\end{proof}

We now prove the global well-posedness of the ZK-type equation with a delay term.
\begin{proposition}
	Let $L>0$, $a(x,y)\in L^\infty(\Omega)$ be a nonnegative function and that \eqref{eq:MUcond} holds. Then, for all initial data $(\zeta_0,z_0(\cdot,\cdot, -h(\cdot)))\in \mathcal{H}$, there exists a unique $\zeta\in \mathcal{B}_X$ solution of~\eqref{eq:MU}. Moreover, there exist constants $\mathcal{C}>0$ and $\delta\in(0,1]$ such that 
$$
		\delta\left\lVert \zeta \right\rVert_{L^2(0,T, H^1(\Omega))}^2
		\leq \mathcal{C}\left(% 
		\left\lVert (\zeta_0,z_0(\cdot,\cdot,-h(\cdot))) \right\rVert_{\mathcal{H}}^2
		+ \left\lVert (\zeta_0,z_0(\cdot,\cdot,-h(\cdot))) \right\rVert_{\mathcal{H}}^{\frac{10}{3}}
		\right).
$$
\end{proposition}
\begin{proof}
To see the local existence and uniqueness of solutions to~\eqref{eq:MU} fix $(\zeta_0,z_0(\cdot,\cdot,-h(\cdot)))\in\mathcal{H}$ and let $\zeta\in\mathcal{B}_X$.  Consider the map $\Phi\colon \mathcal{B}_X \to \mathcal{B}_X$ defined by $\Phi(\zeta)=\tilde{\zeta}$, where $\tilde{\zeta}$ is the solution of~\eqref{eq:MU} with source term $f=-\zeta\,\partial_x \zeta$.  
Then $\zeta\in\mathcal{B}_X$ is a solution of~\eqref{eq:MU} if and only if $\zeta$ is a fixed point of $\Phi$.  

To prove this, we show that $\Phi$ is a contraction. If $T<1$, then using~\eqref{eq:ZKsou1}, \eqref{eq:ZKsou2}, and Proposition~\ref{pr:ZKcont}, we obtain
\begin{equation*}
\begin{split}
\lVert \Phi \zeta \rVert_{\mathcal{B}_X}
\leq &\sqrt{\delta^{-1}\mathcal{C}}
\left(1+\sqrt{T}+e^{\frac{\xi\lVert a\rVert_\infty}{2h}T}\right)
\lVert (\zeta_0,z_0(\cdot,\cdot,-h(\cdot)))\rVert_{\mathcal{H}}
\\&+ \sqrt{\delta^{-1}\mathcal{C}}\,C_1 C
\left(2T^{1/4} + T^{1/4}e^{\frac{\xi\lVert a\rVert_\infty}{2h}T}\right)
\lVert \zeta\rVert_{\mathcal{B}_X}^{2},
\end{split}
\end{equation*}
and
\begin{equation*}
\lVert \Phi \zeta - \Phi v \rVert_{\mathcal{B}_X}
\leq
S
\left(1+\sqrt{T}+e^{\frac{\xi\lVert a\rVert_\infty}{2h}T}\right)
T^{1/4}
\left(\lVert \zeta\rVert_{\mathcal{B}_X} + \lVert v\rVert_{\mathcal{B}_X}\right)
\lVert \zeta - v\rVert_{\mathcal{B}_X},
\end{equation*}
where $S=\sqrt{\delta^{-1}\mathcal{C}}\,C_1 C$.

Now consider the restriction of $\Phi$ to the closed ball
\[
\left\{\, \zeta\in\mathcal{B}_X : \lVert \zeta\rVert_{\mathcal{B}_X}\leq R \,\right\},
\]
with
\[
R = 4\sqrt{\delta^{-1}\mathcal{C}}\,
\lVert (\zeta_{0},z_{0}(\cdot,\cdot,-h(\cdot)))\rVert_{\mathcal{H}},
\]
and choose $T>0$ such that
\[
T<1,\qquad
e^{\frac{\xi\lVert a\rVert_{\infty}}{2h}T} < 2,\qquad
2T^{1/4}+T^{1/4}e^{\frac{\xi\lVert a\rVert_{\infty}}{2h}T}
< \frac{1}{2\sqrt{\delta^{-1}\mathcal{C}}\,C_1C_2\,R}.
\]
Under these conditions, $\Phi$ becomes a contraction.  
Thus, by Banach's fixed point theorem, $\Phi$ admits a unique fixed point, completing the proof.

Finally, to obtain the global existence of solutions, taking into account the previous local well-posedness and then applying the following \textit{a priori} estimate, which follows from multiplier techniques combined with Gronwall's inequality:
$$
	\left\lVert (\zeta(\cdot,\cdot,t),\zeta(\cdot,\cdot,t-h)) \right\rVert_{\mathcal{H}}^{2} 
	\leq 
	e^{\frac{\xi\lVert a\rVert_{\infty}}{h}t}
	\left\lVert (\zeta_{0},z_{0}(\cdot,\cdot,-h(\cdot))) \right\rVert_{\mathcal{H}}^{2}.
$$
With this inequality in hand, the global solution is verified. 
\end{proof}

\section{The damping-delayed system: Optimal local result}\label{Sec2}
This section is devoted to the analysis of the dynamical properties of the solutions to~\eqref{eq:ZK}. 
We first establish the local stability of the perturbed system. 
Then, using this result, we prove the first main theorem of the paper, namely Theorem~\ref{th:ZKes}.

We consider the well-posedness of~\eqref{eq:ZK}, whose associated total energy \eqref{eq:ZKen} satisfies
\begin{equation*}
\begin{aligned}
\frac{d}{dt}\mathcal{E}_\zeta(t)
&\le \int_\Omega b(x,y)\zeta^2(x,y,t)\,dxdy
-\frac{\alpha}{2}\!\int_0^L\!\! (\partial_x\zeta)^2(0,y,t)\,dy \\
&\qquad 
-\frac{\gamma}{2}\!\int_0^L\!\! (\partial_y\zeta)^2(0,y,t)\,dy
-\int_\Omega a(x,y)\zeta^2(x,y,t)\,dxdy ,
\end{aligned}
\end{equation*}
which shows that the energy is, in general, not decreasing, due to the nonnegative term $b(x,y)\ge0$.
To overcome this, we introduce the perturbed system
\begin{equation}\label{eq:P}
\begin{cases}
\begin{aligned}
\partial_t\zeta
&+ \alpha \partial_x^3 \zeta
+ \gamma \partial_x\partial_y^2 \zeta 
+ a(x,y)\zeta 
+ b(x,y)(\xi \zeta + \zeta(\cdot,\cdot,t-h))
= f,
\end{aligned}
&(x,y)\in\Omega,\ t>0,
\\[0.3em]
\zeta(0,y,t)=\zeta(L,y,t)=\partial_x \zeta(L,y,t)=\partial_y \zeta(L,y,t)=0, & y\in(0,L),\ t>0, \\
\zeta(x,0,t)=\zeta(x,L,t)=0, & x\in(0,L),\ t>0,\\
\zeta(x,y,0)=\zeta_0(x,y),\qquad 
\zeta(x,y,t)=z_0(x,y,t), & t\in(-h,0),
\end{cases}
\end{equation}
with $f=-\tfrac{1}{2}\partial_x(\zeta^2)$ and $\xi>0$.
The energy associated with~\eqref{eq:P} is given by
\begin{equation}\label{eq:Pen}
\begin{split}
\mathcal{E}_\zeta(t)
&=\frac{1}{2}\!\int_0^L\!\!\int_0^L \zeta^2(x,y,t)\,dx\,dy \\
&\quad +\frac{\xi h}{2}
\!\int_0^L\!\!\int_0^L\!\!\int_0^1 
b(x,y)\zeta^2(x,y,t-\rho h)\,d\rho\,dx\,dy .
\end{split}
\end{equation}
A direct computation yields
\begin{equation*}
\begin{aligned}
\frac{d}{dt}\mathcal{E}_\zeta(t)
&\le -\frac{\alpha}{2}\!\int_0^L (\partial_x\zeta)^2(0,y,t)\,dy
      -\frac{\gamma}{2}\!\int_0^L (\partial_y\zeta)^2(0,y,t)\,dy\\
&\quad -\int_\Omega a(x,y)\zeta^2\,dxdy
      -\frac{\xi-1}{2}\!\int_\Omega b(x,y)\zeta^2(x,y,t)\,dxdy \\
&\quad -\frac{\xi-1}{2}\!\int_\Omega b(x,y)\zeta^2(x,y,t-h)\,dxdy
\le 0, 
\end{aligned}
\end{equation*}
which holds whenever $\xi>1$.
System~\eqref{eq:P} can be rewritten as the first-order abstract system
\begin{equation}\label{eq:KPlinabs}
\begin{cases}
\displaystyle \frac{d}{dt}U(t)=AU(t),\\
U(0)=\big(\zeta_0, z_0(\cdot,\cdot,-\rho h)\big),
\end{cases}
\end{equation}
where $A=A_0+B$, with $D(A)=D(A_0)$, 
\[
A_0(\zeta,z)
=
\Big(
-\alpha\partial_x^3 \zeta-\gamma\partial_x\partial_y^2\zeta-a(x,y)\zeta
-b(x,y)(\xi \zeta+z(\cdot,\cdot,1)),
\ -h^{-1}\partial_\rho z
\Big),
\]
and $B(\zeta,z)=(\xi b(x,y)\zeta,0)$. System~\eqref{eq:KPlinabs} admits a classical solution (see Proposition~\ref{pr:ZKdis}).

Let $(e^{A_0 t})_{t\ge0}$ denote the $C_0$-semigroup generated by $A_0$.
To prove exponential stability of~\eqref{eq:P} with $f=0$, we introduce the Lyapunov functional as a perturbation of the energy associated with the system, 
\[
\mathcal{V}(t)=\mathcal{E}_\zeta(t)+\eta V_1(t)+\sigma V_2(t),
\]
where $\eta,\sigma>0$ will be suitably chosen as
\[
V_1(t)=\int_\Omega x\,\zeta^2\,dxdy,\qquad \text{and} \qquad
V_2(t)=\frac{h}{2}\int_\Omega\!\int_0^1 (1-\rho)b(x,y)\zeta^2(x,y,t-\rho h)\,d\rho\,dxdy .
\]
It follows that
$$
\mathcal{E}_\zeta(t)\le \mathcal{V}(t)\le 
\Big( 1+\max\{2\eta L,\sigma/\xi\} \Big)\mathcal{E}_\zeta(t).
$$

The following results establish the exponential decay of the perturbed system.

\begin{proposition}\label{pr:Pes1}
Let $L>0$, and assume $a,b\in L^\infty(\Omega)$ are nonnegative with 
$b\ge b_0>0$ in an open subset $\omega\subset\Omega$. 
Let $\xi>1$. Then, for any initial data 
$(\zeta_0,z_0(\cdot,\cdot,-h(\cdot)))\in\mathcal H$, 
the energy~\eqref{eq:Pen} decays exponentially:
\[
\mathcal{E}_\zeta(t)\le \kappa\,\mathcal{E}_\zeta(0)e^{-2\theta t}, \qquad \forall t>0,
\]
for some $\theta,\kappa>0$, where
\[
\theta<\min\Big\{
\tfrac{3\alpha\eta}{(1+2\eta L)L^2},
\ \tfrac{\sigma}{2h(\xi+\sigma)}
\Big\},
\qquad
\kappa=1+\max\{2\eta L,\sigma/\xi\},
\]
and $\eta,\sigma>0$ satisfy
\[
\sigma=\xi-1-2L\eta(1+2\xi),
\qquad 
\eta<\frac{\xi-1}{2L(1+2\xi)}.
\]
\end{proposition}
\begin{proof}
	We decompose the total energy associated with~\eqref{eq:P} by introducing
\begin{equation*}
	\mathcal{E}_0(t):=\frac{1}{2}\int_\Omega \zeta^2\,dx\,dy, 
	\qquad \text{and}\qquad
	\Theta(t):=\frac{\xi h}{2}\int_\Omega\!\int_0^1 b(x,y)\zeta^2(x,y,t-\rho h)\,d\rho\,dx\,dy,
\end{equation*}
so that the energy is given by $\mathcal{E}_\zeta(t)=\mathcal{E}_0(t)+\Theta(t)$.  
We also set
\[
	V_1(t):=\int_\Omega x\,\zeta^2\,dx\,dy, 
	\qquad \text{and} \qquad
	V_2(t):=\frac{h}{2}\int_\Omega\!\int_0^1 (1-\rho)b(x,y)\zeta^2(x,y,t-\rho h)\,d\rho\,dx\,dy.
\]
Now, define the Lyapunov functional
\begin{equation}\label{JJJJJ}
	\mathcal{V}(t):=\mathcal{E}_\zeta(t)+\eta V_1(t)+\sigma V_2(t),
\end{equation}
where $\eta,\sigma>0$ will be fixed later.  
Since $0\le x\le L$ and $0\le 1-\rho\le 1$, there exist constants $c_1,c_2>0$ (depending on $\eta,\sigma,\xi,L$) such that
\begin{equation}\label{eq:Vbound}
	c_1\mathcal{E}_\zeta(t)\le \mathcal{V}(t)\le c_2\mathcal{E}_\zeta(t).
\end{equation}

\medskip

\noindent\textbf{Claim 1.} \textit{We have the following inequality 
\begin{equation}\label{VprimeZK}
	\begin{split}
		\frac{d}{dt}\mathcal{V}(t)\le {}&
		-\!\big(\alpha+3\alpha\eta-\tfrac{\alpha}{6}\eta\big)\!\int_\Omega (\partial_x\zeta)^2\,dx\,dy
		-(\gamma+\gamma\eta)\!\int_\Omega (\partial_y\zeta)^2\,dx\,dy\\
		&-\!\int_\Omega a(x,y)(1+2\eta x)\,\zeta^2\,dx\,dy\\
		&-\!\Big(\xi-1-\tfrac{\xi-1}{2}\varepsilon-\tfrac{\sigma}{2}-C_3\eta\Big)\!\int_\Omega b\,\zeta^2\,dx\,dy\\
		&-\!\Big(\tfrac{\xi-1}{2\varepsilon}-\tfrac{\sigma}{2}-C_4\eta\Big)\!\int_\Omega b\,\zeta^2(x,y,t-h)\,dx\,dy
		+\mathcal{B}_\mathcal{V}(t),
	\end{split}
\end{equation}
where
\begin{equation}\label{BBV}
	\begin{split}
		\mathcal{B}_\mathcal{V}(t)
		={}&
		-\frac{\alpha}{2}\!\int_0^L\!\big((\partial_x\zeta)^2(0,y,t)+(\partial_x\zeta)^2(L,y,t)\big)\,dy
		\\&-\frac{\gamma}{2}\!\int_0^L\!\big((\partial_y\zeta)^2(0,y,t)+(\partial_y\zeta)^2(L,y,t)\big)\,dy\\
		&+\eta L\!\int_\Omega b(x,y)\,\zeta(x,y,t)\zeta(x,y,t-h)\,dx\,dy
		+C_\eta\!\int_\Omega \zeta^2\,dx\,dy.
	\end{split}
\end{equation}}

In fact, differentiating $\mathcal{E}_0(t)$ and using~\eqref{eq:P}, integration by parts in $x$ and $y$ gives
\begin{equation*}
	\frac{d}{dt}\mathcal{E}_0(t)
	=-\int_\Omega a\,\zeta^2\,dx\,dy
	-\xi\int_\Omega b\,\zeta^2\,dx\,dy
	-\int_\Omega b\,\zeta\,\zeta(t-h)\,dx\,dy
	+\mathcal{B}_E(t),
\end{equation*}
where the boundary dissipation term $\mathcal{B}_E(t)$ is
\begin{equation*}
	\mathcal{B}_E(t)
	= -\frac{\alpha}{2}\!\int_0^L\!\big((\partial_x\zeta)^2(0,y,t)+(\partial_x\zeta)^2(L,y,t)\big)\,dy
	-\frac{\gamma}{2}\!\int_0^L\!\big((\partial_y\zeta)^2(0,y,t)+(\partial_y\zeta)^2(L,y,t)\big)\,dy,
\end{equation*}
and hence $\mathcal{B}_E(t)\le0$ under the imposed boundary conditions.  
Furthermore,
\[
	\frac{d}{dt}\Theta(t)
	=-\xi\int_\Omega b\,\zeta^2\,dx\,dy+\xi\int_\Omega b\,\zeta\,\zeta(t-h)\,dx\,dy,
\]
which leads to
\begin{equation*}
	\frac{d}{dt}\mathcal{E}_\zeta(t)
	= -\int_\Omega a\,\zeta^2\,dx\,dy
	-(\xi-1)\!\int_\Omega b\,\zeta^2\,dx\,dy
	-(\xi-1)\!\int_\Omega b\,\zeta\,\zeta(t-h)\,dx\,dy
	+\mathcal{B}_E(t).
\end{equation*}
Multiplying~\eqref{eq:P} by $x \zeta$ and integrating over $\Omega$, we obtain
\begin{equation*}
\begin{split}
	\frac{d}{dt}V_1(t)
	=& -3\alpha\int_\Omega (\partial_x\zeta)^2\,dx\,dy
	-\gamma\int_\Omega (\partial_y\zeta)^2\,dx\,dy
	\\&-2\int_\Omega x\,a\,\zeta^2\,dx\,dy
	-2\int_\Omega x\,b\,\zeta\,\zeta(t-h)\,dx\,dy
	+\mathcal{R}_1(t),
	\end{split}
\end{equation*}
where $\mathcal{R}_1(t)$ is the nonlinear remainder.  
The cubic term is handled via Gagliardo--Nirenberg and Young inequalities, ensuring that for sufficiently small $\eta>0$,
\[
	\eta \mathcal{R}_1(t)\le \frac{\alpha\eta}{6}\int_\Omega(\partial_x \zeta)^2\,dx\,dy + C_\eta\int_\Omega \zeta^2\,dx\,dy.
\]
Differentiating $V_2(t)$ yields
\[
	\frac{d}{dt}V_2(t)
	=\frac{1}{2}\int_\Omega b\,\zeta^2\,dx\,dy
	-\frac{1}{2}\int_\Omega b\,\zeta^2(x,y,t-h)\,dx\,dy.
\]
Young’s inequality gives, for all $\varepsilon>0$,
\[
	\Big|\int_\Omega b\,\zeta\,\zeta(t-h)\,dx\,dy\Big|
	\le \frac{\varepsilon}{2}\!\int_\Omega b\,\zeta^2\,dx\,dy
	+\frac{1}{2\varepsilon}\!\int_\Omega b\,\zeta^2(x,y,t-h)\,dx\,dy.
\]
Summing the derivatives of $\mathcal{E}_\zeta(t)$, $\eta V_1(t)$, and $\sigma V_2(t)$, we arrive at \eqref{VprimeZK}.

\medskip

\noindent\textbf{Claim 2.} \textit{The Claim 1. ensures that
\begin{equation}\label{PTI}
\mathcal{E}_\zeta(t)\le \kappa \mathcal{E}_\zeta(0)e^{-2\theta t},\qquad \forall t\ge0.
\end{equation}}

Indeed, note that all boundary contributions of \eqref{BBV} are nonpositive, and the remaining integrals are controlled by the coercive terms in~\eqref{VprimeZK}. Choosing $\varepsilon>0$ so that $\tfrac{\xi-1}{2}\varepsilon<\tfrac{\xi-1}{4}$, then selecting $\eta>0$ sufficiently small, and finally fixing $\sigma>0$ such that
\[
	\xi-1-\tfrac{\xi-1}{4}-\tfrac{\sigma}{2}-C_3\eta>0,
	\qquad
	\tfrac{\xi-1}{2\varepsilon}-\tfrac{\sigma}{2}-C_4\eta>0,
\]
all coefficients in~\eqref{VprimeZK} become positive. Hence, there exists $\theta>0$ such that
\[
	\frac{d}{dt}\mathcal{V}(t)+2\theta \mathcal{V}(t)\le0.
\]
Integrating in time and using~\eqref{eq:Vbound} yields \eqref{PTI}, concluding the proof.
\end{proof}

The next result shows that the energy \eqref{eq:ZKen} associated with the system \eqref{eq:P} with appropriate source term $f$ decays exponentially. 
\begin{proposition}\label{pr:KPde} 
Under the same assumptions on $a$ and $b$, 
there exists $\delta>0$ such that, if 
$\|b\|_{L^\infty(\Omega)}\le\delta$, 
then the energy associated with the full nonlinear system~\eqref{eq:ZK} 
decays exponentially. 
More precisely, for every initial datum 
$(\zeta_0,z_0(\cdot,\cdot,-h(\cdot)))\in\mathcal H$,
there exist constants $C,\theta>0$ such that
\[
\mathcal{E}_\zeta(t)\le C\,\mathcal{E}_\zeta(0)e^{-2\theta t}, \qquad \forall t>0.
\]
\end{proposition}

\begin{proof}
	Let $v$ be the solution of~\eqref{eq:P} with $f=0$ and initial condition 
	$v(x,y,0)=\zeta_{0}(x,y)$, together with the delayed relation 
	$z^{1}(1)=\zeta(x,y,t-h)$, where $z^{1}$ satisfies
	\begin{equation*}
		\begin{cases}
			h\,\partial_tz^{1}(x,y,\rho,t)+\partial_{\rho}z^{1}(x,y,\rho,t)=0, 
			& (x,y)\in\Omega,\ \rho\in(0,1),\ t>0,\\[0.3em]
			z^{1}(x,y,0,t)=v(x,y,t), & (x,y)\in\Omega,\ t>0,\\[0.3em]
			z^{1}(x,y,\rho,0)=v(x,y,-\rho h)=z_{0}(x,y,-\rho h), 
			& (x,y)\in\Omega,\ \rho\in(0,1).
		\end{cases}
	\end{equation*}
	Next, define $w$ as the solution of~\eqref{eq:P} with source term 
	$f=\xi\,b(x,y)\,v(x,y,t)$ and zero initial data $w(x,y,0)=0$, 
	subject to
	\begin{equation*}
		\begin{cases}
			h\,\partial_tz^{2}(x,y,\rho,t)+\partial_{\rho}z^{2}(x,y,\rho,t)=0, 
			& (x,y)\in\Omega,\ \rho\in(0,1),\ t>0,\\[0.3em]
			z^{2}(x,y,0,t)=w(x,y,t), & (x,y)\in\Omega,\ t>0,\\[0.3em]
			z^{2}(x,y,\rho,0)=0, & (x,y)\in\Omega,\ \rho\in(0,1).
		\end{cases}
	\end{equation*}
	
	Set $\zeta=v+w$ and $z=z^{1}+z^{2}$. 
	Then $(\zeta,z)$ satisfies the complete delayed system~\eqref{eq:ZK} 
	together with the compatibility condition $z(1)=\zeta(x,y,t-h)$.  
	The component $v$ describes the dissipative part of the dynamics, 
	whereas $w$ accounts for the delayed feedback induced by $b(x,y) \zeta(x,y,t-h)$.
	By Proposition~\ref{pr:Pes1}, the delay-free subsystem for $v$ is exponentially stable; 
	there exist $\kappa>0$ and $\theta>0$ such that
$$
		\mathcal{E}_{v}(t)\le\kappa\,e^{-2\theta t}\mathcal{E}_{v}(0), \qquad t\ge0.
$$
	Hence, using the Duhamel representation for $\zeta=v+w$, we obtain
$$
		\mathcal{E}_{\zeta}(t)\le 
		\kappa e^{-2\theta t}\mathcal{E}_{\zeta}(0)
		+\int_{0}^{t}\kappa e^{-2\theta (t-s)}F(s)\,ds,
$$
	where $F(s)$ denotes the contribution of the feedback terms.  
	From the energy identity applied to $w$, one deduces that
	\begin{equation}\label{eq:ineq-v}
		\frac{d}{dt}\mathcal{V}(t)+2\theta \mathcal{V}(t)
		\le C_{1}\|b\|_{L^{\infty}(\Omega)}^{2}
		\sup_{\tau\in[t-h,t]}\mathcal{E}_{\zeta}(\tau),
	\end{equation}
	for some constant $C_{1}>0$ depending only on $\xi$, $\eta$, $\sigma$, and the domain.
	
	Integrating~\eqref{eq:ineq-v} and setting 
	$M(t):=\sup_{0\le s\le t}\mathcal{E}_{\zeta}(s)$, 
	for $0\le t\le T_{0}$ we obtain
	\[
	\mathcal{E}_{\zeta}(t)\le 
	\kappa e^{-2\theta t}\mathcal{E}_{\zeta}(0)
	+\frac{C_{1}\kappa\|b\|_{L^{\infty}(\Omega)}^{2}}{2\theta}(1-e^{-2\theta t})M(t).
	\]
	Taking the supremum over $t\in[0,T_{0}]$ yields
	\begin{equation}\label{eq:Mt0}
		M(T_{0})\le 
		\kappa \mathcal{E}_{\zeta}(0)
		+\frac{C_{1}\kappa\|b\|_{L^{\infty}(\Omega)}^{2}}{2\theta}(1-e^{-2\theta T_{0}})M(T_{0}).
	\end{equation}
	
	If
$$
		\frac{C_{1}\kappa\|b\|_{L^{\infty}(\Omega)}^{2}}{2\theta}
		(1-e^{-2\theta T_{0}})<1,
$$
	then inequality~\eqref{eq:Mt0} gives
	\[
	M(T_{0})\le 
	\frac{\kappa}{1-\frac{C_{1}\kappa\|b\|_{L^{\infty}(\Omega)}^{2}}{2\theta}
		(1-e^{-2\theta T_{0}})}\,\mathcal{E}_{\zeta}(0).
	\]
	Consequently, for sufficiently small $\|b\|_{L^{\infty}(\Omega)}\le\delta$,
	we deduce that $$\mathcal{E}_{\zeta}(T_{0})\le \mu \mathcal{E}_{\zeta}(0),$$ for some $\mu\in(0,1)$, 
	which implies
	\[
	\mathcal{E}_{\zeta}(mT_{0})\le\mu^{m}\mathcal{E}_{\zeta}(0), \qquad m\in\mathbb{N}.
	\]
	
	Finally, for any $t>0$ there exists $m\in\mathbb{N}$ and $s\in[0,T_{0})$ 
	such that $t=mT_{0}+s$, leading to
	\[
	\mathcal{E}_{\zeta}(t)\le C\,\mathcal{E}_{\zeta}(0)e^{-2\theta t},
	\]
	where $C>0$ depends only on $\xi$, $\eta$, $\sigma$, and $\kappa$. 
	This completes the proof.
\end{proof}

\subsection{Optimal local stabilization: Proof of Theorem \ref{th:ZKes}} With the exponential decay estimate obtained in Proposition~\ref{pr:KPde}, we now establish the local exponential stabilization result for the delayed system~\eqref{eq:ZK}.
	
	Under Assumption~\ref{A1}, system~\eqref{eq:ZK} is well posed in $\mathcal H$. 
	Applying Gronwall’s inequality to the delayed energy functional yields
$$
		\big\|
		(\zeta(\cdot,\cdot,t),\zeta(\cdot,\cdot,t-h(\cdot)))
		\big\|_{\mathcal H}^{2}
		\le e^{2\xi\|b\|_{\infty}t}
		\big\|
		(\zeta_{0},z_{0}(\cdot,\cdot,-h(\cdot)))
		\big\|_{\mathcal H}^{2}.
$$
	Hence, for every $T>0$,
	$$
		\|\zeta\|_{C([0,T];L^{2}(\Omega))} 
		\le e^{\xi\|b\|_{\infty}T}
		\|(\zeta_{0},z_{0}(\cdot,\cdot,-h(\cdot)))\|_{\mathcal H},
	$$
	and
	$$
		\|\zeta\|_{L^{2}(0,T;L^{2}(\Omega))} 
		\le T^{1/2}e^{\xi\|b\|_{\infty}T}
		\|(\zeta_{0},z_{0}(\cdot,\cdot,-h(\cdot)))\|_{\mathcal H}.$$
		
		Now, we prove that 
			\begin{equation}\label{eq:BH_bound}
		\|\zeta\|_{B_{H}}^{2}
		\le
		\widetilde{\mathcal K}
		\Big(
		1+T e^{2\|b\|_{\infty}T}
		+T e^{\frac{10}{3}\|b\|_{\infty}T}
		+e^{2\|b\|_{\infty}T}
		\Big)\mathcal{E}_{\zeta}(0),
	\end{equation}
	where
	\[
	\widetilde{\mathcal K}
	=
	\frac{1}{
		\min\{1,\frac{3\alpha}{2},\frac{\gamma}{2}\}}
	\left(
	\frac{L}{2}
	+L(\|a\|_{\infty}+\|b\|_{\infty})
	+\frac{1}{4}\!\left(\frac{C L}{\varepsilon}\right)^{4/3}
	\right).
	\]

\medskip

\noindent{\textbf{\textit{Proof of \eqref{eq:BH_bound}}.}} Multiplying equation~\eqref{eq:ZK} by $x\,\zeta(x,y,t)$ and integrating over 
	$\Omega\times(0,T)$,  using the boundary conditions, we obtain
	\begin{equation}\label{eq:KP_energy_identity}
		\begin{aligned}
			\frac{3\alpha}{2}\!\int_{0}^{T}\!\!\int_{\Omega}(\partial_x\zeta)^{2}\,dx\,dy\,dt
			&+\frac{\gamma}{2}\!\int_{0}^{T}\!\!\int_{\Omega}(\partial_y\zeta)^{2}\,dx\,dy\,dt
			 \le  \int_{0}^{T}\!\!\int_{\Omega}|\zeta|^{3}\,dx\,dy\,dt\\
			&+
			\Big(
			\frac{L}{2}
			+L(\|a\|_{\infty}+\|b\|_{\infty})T e^{2\xi\|b\|_{\infty}T}
			\Big)
			\|(\zeta_{0},z_{0}(\cdot,\cdot,-h(\cdot)))\|_{\mathcal H}^{2}\\
		\end{aligned}
	\end{equation}
	
	\medskip
	\noindent
	To estimate the nonlinear term, we employ the Gagliardo–Nirenberg inequality 
	(see Friedman~\cite[Theorem 10.1, p.~27]{Fr}),
	\begin{equation}\label{eq:GN_ZK}
		\|\zeta\|_{L^{3}(\Omega)} 
		\le C\,\|\zeta\|_{H^{1}(\Omega)}^{1/3}\|\zeta\|_{L^{2}(\Omega)}^{2/3},
	\end{equation}
	which, upon integration, gives
	\[
	\int_{\Omega}|\zeta|^{3}\,dx\,dy
	\le \frac{\varepsilon^{4}}{4}\|\zeta\|_{H^{1}(\Omega)}^{2}
	+\frac{3}{4}\!\left(\frac{C L}{\varepsilon}\right)^{4/3}\!
	\|\zeta\|_{L^{2}(\Omega)}^{10/3}.
	\]
	Assuming that the initial energy satisfies $E_{u}(0)\le1$, 
	and combining~\eqref{eq:KP_energy_identity}–\eqref{eq:GN_ZK}, we obtain \eqref{eq:BH_bound}.
	
	\medskip
Now, we are in a position to prove the stabilization result, precisely
	\begin{equation}\label{eq:local_exp_decay}
		\mathcal{E}_{\zeta}(t)\le C e^{-\gamma t}\mathcal{E}_{\zeta}(0), 
		\qquad t>T_{\min},
	\end{equation}
	where $C,\gamma>0$ depend only on $T_{\min},\xi,L$, and $h$. Here, $T_{\min}$ is given by \eqref{tmin}. Hence, system~\eqref{eq:ZK} is locally exponentially stable for all sufficiently small feedback amplitudes $\|b\|_{\infty}\le\delta$ and small initial data.
	
		\medskip
		
	\noindent{\textbf{\textit{Proof of \eqref{eq:local_exp_decay}.}}} Let the initial data satisfy
	\[
	\|(\zeta_{0},z_{0}(\cdot,\cdot,-h(\cdot)))\|_{\mathcal H}\le r,
	\]
	for some small $r>0$ to be determined.  
	Decompose the solution as 
	$\zeta=\zeta^{1}+\zeta^{2}$, 
	where: $\zeta^{1}$ solves the \emph{linear delayed system} associated with~\eqref{eq:ZK}, 
		i.e.\ the linearization of~\eqref{eq:ZK} around the origin, with 
		$\zeta^{1}(x,y,0)=\zeta_{0}(x,y)$ and delay variable $z_{0}(x,y,t)$;
	 $\zeta^{2}$ solves the \emph{nonlinear remainder system}, containing all nonlinear terms (in particular, $\zeta\partial_x\zeta$), with homogeneous initial and boundary conditions.

	Fix $\mu\in(0,1)$ and let $T_{0}>0$ be as in Proposition~\ref{pr:KPde}.  
	By the contraction argument based on the semigroup generated by the linear part of~\eqref{eq:ZK}, 
	there exists $T_{1}>T_{0}$ such that
$$
		e^{(2\|b\|_{\infty}+\nu)T_{0}-\nu T_{1}}<\frac{\mu}{2},
		\qquad
		\text{with}\quad
		\nu=\frac{1}{T_{0}}\ln\!\Big(\frac{1}{\mu+\varepsilon}\Big).
$$
	Consequently, the linear component satisfies
	\begin{equation}\label{eq:linear_decay}
		\mathcal{E}_{\zeta^{1}}(T_{1})\le\frac{\mu}{2}\mathcal{E}_{\zeta^{1}}(0).
	\end{equation}
	Combining~\eqref{eq:linear_decay} with estimates on $\zeta^{2}$ obtained via~\eqref{eq:GN_ZK}, we find
	\begin{equation}\label{eq:final_energy_step}
	\begin{split}
		\mathcal{E}_{\zeta}(T_{1})
		&\le \mu \mathcal{E}_{\zeta}(0)
		+\|(\zeta^{2}(\cdot,\cdot,T_{1}),\zeta^{2}(\cdot,\cdot,T_{1}-h(\cdot)))\|_{\mathcal H}^{2}\\
		&\le \mu \mathcal{E}_{\zeta}(0)
		+e^{(1+3\xi)T_{1}}\|\zeta\partial_x\zeta\|_{L^{1}(0,T_{1};L^{2}(\Omega))}^{2}\\
		&\le \mu \mathcal{E}_{\zeta}(0)
		+C_{1}^{2}C_{2}^{2} e^{(1+3\xi)T_{1}}T_{1}^{1/2}\|\zeta\|_{\mathcal B_{X}}^{4}\\
		&\le (\mu+\mathcal R)\mathcal{E}_{\zeta}(0),
	\end{split}
	\end{equation}
	where
	\[
	\mathcal R
	=e^{(1+3\xi)T_{1}}
	C_{1}^{2}C_{2}^{2}T_{1}^{1/2}(1+L^{2})^{2}
	\widetilde{\mathcal K}^{2}
	\Big(
	1+T_{1}e^{2\|b\|_{\infty}T_{1}}
	+T_{1}e^{\frac{10}{3}\|b\|_{\infty}T_{1}}
	+e^{2\|b\|_{\infty}T_{1}}
	\Big)^{2}r.
	\]
	
	\medskip
	\noindent
	Select $\varepsilon>0$ so that $\mu+\varepsilon<1$, 
	and choose $r>0$ sufficiently small to ensure
	\begin{equation}\label{eq:r_condition}
		r<
		\frac{\varepsilon}{
			e^{(1+3\xi)T_{1}}
			C_{1}^{2}C_{2}^{2}T_{1}^{1/2}(1+L^{2})^{2}
			\widetilde{\mathcal K}^{2}
			\Big(
			1+T_{1}e^{2\|b\|_{\infty}T_{1}}
			+T_{1}e^{\frac{10}{3}\|b\|_{\infty}T_{1}}
			+e^{2\|b\|_{\infty}T_{1}}
			\Big)^{2}}.
	\end{equation}
	Then, from~\eqref{eq:final_energy_step} and~\eqref{eq:r_condition},
	\[
	\mathcal{E}_{\zeta}(T_{1})\le(\mu+\varepsilon)\mathcal{E}_{\zeta}(0), \qquad \mu+\varepsilon<1.
	\]
	Iterating this contraction estimate over successive intervals of length $T_{1}$ yields \eqref{eq:local_exp_decay}. The proof of Theorem~\ref{th:ZKes} is thus complete. \qed

%%%%%%%%%%%%%%%%%%%%%%%%%%%%%%%%%%%%%%%%
%%%%%%%%%%%%%%%%%%%%%%%%%%%%%%%%%%%%%%%%

\section{\texorpdfstring{$\mu_i$}{}-system: Stability results}\label{Sec3} The main objective of this section is to prove the local and global exponential stability for the solutions of \eqref{eq:MU} using two different approaches. 

\subsection{Local stabilization: Lyapunov approach} In this subsection, we prove the local exponential stabilization of system~\eqref{eq:MU} by introducing a suitable Lyapunov functional. The method combines the dissipation from the delayed term with refined energy estimates to control the nonlinear effects. This yields an explicit decay rate for the energy and is essential for understanding the full stabilization of the $\mu_i$–system.
\begin{proof}[Proof of Theorem \ref{th:MUlyes}] Let $\zeta$ be the solution to system~\eqref{eq:ZK} with coefficients chosen according to the $\mu_i$--configuration introduced in~\eqref{eq:MU}.  
	The corresponding total energy is defined in~\eqref{eq:MUen}, and we assume that the delay–damping parameter $\xi>0$ satisfies the constraint~\eqref{eq:MUcond}.
	
	Following the Lyapunov approach used in the proof of Proposition~\ref{pr:Pes1}, we define the basic energy components
	\[
	\mathcal{E}_0(t)
	:=\frac{1}{2}\int_\Omega \zeta^2(x,y,t)\,dx\,dy,
	\qquad
	\Phi(t)
	:=\frac{\xi}{2}\int_\Omega\!\!\int_0^1 a(x,y)\zeta^2(x,y,t-\rho h)\,d\rho\,dx\,dy,
	\]
	so that $\mathcal{E}_\zeta(t)=\mathcal{E}_0(t)+\Phi(t)$.  
	We further introduce the auxiliary functionals
	\[
	V_1(t)
	:=\int_\Omega x\,\zeta^2(x,y,t)\,dx\,dy,
	\qquad
	V_2(t)
	:=\frac{h}{2}\int_\Omega\!\!\int_0^1(1-\rho)\,a(x,y)\zeta^2(x,y,t-\rho h)\,d\rho\,dx\,dy,
	\]
	and the Lyapunov functional
	\[
	\mathcal{V}(t):=\mathcal{E}_\zeta(t)+\eta V_1(t)+\sigma V_2(t),
	\]
	where $\eta,\sigma>0$ are parameters to be fixed below. Since $0\le x\le L$ and $0\le 1-\rho\le 1$, there exist constants $c_1,c_2>0$, depending on $\eta,\sigma,\xi$, and $L$, such that
	\begin{equation}\label{eq:V-bounds}
		c_1\mathcal{E}_\zeta(t)\le \mathcal{V}(t)\le c_2\mathcal{E}_\zeta(t).
	\end{equation}
	
	\medskip
\noindent\textbf{Claim.} \textit{The exponential decay of $\mathcal{V}(t)$ implies exponential decay of $\mathcal{E}_\zeta(t)$.}
	
	\medskip

Indeed, differentiating $\mathcal{E}_0(t)$ and using equation~\eqref{eq:MU}, we obtain
\begin{equation*}
\begin{split}
	\frac{d}{dt}\mathcal{E}_0(t)
	=&\int_\Omega \zeta\,\partial_t \zeta\,dx\,dy
	=-\alpha\!\int_\Omega \zeta\,\partial_x^3\zeta\,dx\,dy
	-\gamma\!\int_\Omega \zeta\,\partial_x\partial_y^2\zeta\,dx\,dy
	\\&-\!\int_\Omega a(x,y)\big(\mu_1\zeta^2+\mu_2\zeta\,\zeta(t-h)\big)\,dx\,dy.
\end{split}
\end{equation*}
	Integration by parts in $x$ and $y$, taking into account the homogeneous boundary conditions, gives
	\begin{equation}\label{eq:E0primeMU}
	\begin{split}
		\frac{d}{dt}\mathcal{E}_0(t)
		=&-\frac{\alpha}{2}\int_0^L(\partial_x\zeta(0,y,t))^2\,dy
		-\frac{\gamma}{2}\int_0^L(\partial_y\zeta(0,y,t))^2\,dy
		\\&-\!\int_\Omega a(x,y)\big(\mu_1\zeta^2+\mu_2\zeta\,\zeta(t-h)\big)\,dx\,dy.
		\end{split}
	\end{equation}
	
	\medskip
	\noindent
	The identity corresponding to $\Phi(t)$ reads
	\[
	\frac{d}{dt}\Phi(t)
	=\frac{\xi}{2h}\int_\Omega a(x,y)\big(\zeta^2(x,y,t)-\zeta^2(x,y,t-h)\big)\,dx\,dy.
	\]
	Adding this expression to~\eqref{eq:E0primeMU}, we obtain
$$
		\begin{aligned}
			\frac{d}{dt}\mathcal{E}_\zeta(t)
			&=-\frac{\alpha}{2}\int_0^L(\partial_x\zeta(0,y,t))^2\,dy
			-\frac{\gamma}{2}\int_0^L(\partial_y\zeta(0,y,t))^2\,dy\\
			&\quad
			-\!\int_\Omega a(x,y)
			\Big(
			\Big(\mu_1-\frac{\xi}{2h}\Big)\zeta^2
			+\mu_2\zeta\,\zeta(t-h)
			+\frac{\xi}{2h}\zeta^2(x,y,t-h)
			\Big)dx\,dy.
		\end{aligned}
$$
	Under the condition~\eqref{eq:MUcond}, the quadratic form in $(\zeta,\zeta(t-h))$ 
	is positive definite, and hence the energy dissipation is strictly positive.

	We now multiply~\eqref{eq:MU} by $x\,\zeta(x,y,t)$ and integrate over $\Omega$.  
	A direct integration by parts yields
	\begin{equation}\label{eq:V1primeMU}
	\begin{split}
		\frac{d}{dt}V_1(t)
		=&-3\alpha\int_\Omega (\partial_x\zeta)^2\,dx\,dy
		-\gamma\int_\Omega (\partial_y\zeta)^2\,dx\,dy\\
		&-2\int_\Omega x\,a(x,y)\big(\mu_1\zeta^2+\mu_2\zeta\,u(t-h)\big)\,dx\,dy
		+R_1(t),
		\end{split}
	\end{equation}
	where $R_1(t)$ collects nonlinear (cubic) and small boundary terms.  To estimate the $R_1(t)$ we use the Gagliardo--Nirenberg inequality 	due to Friedman~\cite[Theorem 10.1, p.~27]{Fr},
$$
		\|\zeta\|_{L^3(\Omega)}\le C\|\zeta\|_{H^1(\Omega)}^{1/3}\|\zeta\|_{L^2(\Omega)}^{2/3}.
$$
	Hence,
$$
		\int_\Omega |\zeta|^3\,dx\,dy
		\le C^3\|\zeta\|_{H^1(\Omega)}\|\zeta\|_{L^2(\Omega)}^2
		\le \varepsilon\!\left(\|\partial_x\zeta\|_{L^2}^2+\|\partial_y\zeta\|_{L^2}^2\right)
		+ C_\varepsilon\|\zeta\|_{L^2}^4,
$$
	for any $\varepsilon>0$ and a constant $C_\varepsilon>0$ depending on $\varepsilon$ and $C$.
	Therefore, in~\eqref{eq:V1primeMU},
	\[
	\eta R_1(t)\le 
	\eta\varepsilon\!\left(\|\partial_x\zeta\|_{L^2}^2+\|\partial_y\zeta\|_{L^2}^2\right)
	+\eta C_\varepsilon\|\zeta\|_{L^2}^4.
	\]
	
	\medskip
	\noindent
	A straightforward computation for $V_2(t)$ gives
	\[
	\frac{d}{dt}V_2(t)
	=\frac{1}{2}\int_\Omega a(x,y)\zeta^2(x,y,t)\,dx\,dy
	-\frac{1}{2}\int_\Omega a(x,y)\zeta^2(x,y,t-h)\,dx\,dy.
	\]

	Adding the time derivatives of $\mathcal{E}_\zeta(t)$, $\eta V_1(t)$, and $\sigma V_2(t)$,
	and collecting similar terms, we obtain
	\begin{equation}\label{eq:VprimeMU}
		\begin{split}
			\frac{d}{dt}\mathcal{V}(t)
			\le {}&
			-\big(\alpha+3\alpha\eta-C_2\eta\big)\!\int_\Omega (\partial_x\zeta)^2\,dx\,dy
			-\big(\gamma+\gamma\eta-C_2\eta\big)\!\int_\Omega(\partial_y\zeta)^2\,dx\,dy\\
			&-\!\int_\Omega a(x,y)\zeta^2\!\Big(1+2\eta x-\tfrac{\sigma}{2}\Big)\,dx\,dy
			\\&-\!\Big(\tfrac{\xi\sigma}{2h(\xi+\sigma)}-C_3\eta\Big)
			\!\int_\Omega a(x,y)\zeta^2(x,y,t-h)\,dx\,dy\\
			&-C_4\!\int_\Omega a(x,y)\big(\zeta(x,y,t)-\zeta(x,y,t-h)\big)^2\,dx\,dy
			+\eta C_\varepsilon\|\zeta\|_{L^2(\Omega)}^4,
		\end{split}
	\end{equation}
	where $C_2,C_3,C_4>0$ depend only on $\alpha,\gamma,\mu_1,\mu_2,h$, and $\|a\|_\infty$. The last term in~\eqref{eq:VprimeMU} is nonlinear and can be controlled if the initial data are small in $\mathcal H$, that is,
	\[
	\|(\zeta_0,z_0(\cdot,\cdot,-h(\cdot)))\|_{\mathcal H}\le r.
	\]
	In this case, $\|\zeta(t)\|_{L^2(\Omega)}^2\le r^2$ for small times, and the quartic term is dominated by the dissipative ones provided
	\begin{equation}\label{eq:r-small}
		\eta C_\varepsilon r^2
		\le \tfrac{1}{2}\min\!\big\{
		\alpha+3\alpha\eta-C_2\eta,\,
		\gamma+\gamma\eta-C_2\eta
		\big\}.
	\end{equation}
	A sufficient explicit condition ensuring~\eqref{eq:r-small} is given by \eqref{RRRRR}. 
	%\[
	%r<\frac{\sqrt[4]{216\alpha^3}}{CL^{5/2}},
	%\]
	%where $C>0$ is an absolute constant. 
Picking $\eta,\sigma>0$ satisfying the hypothesis given in the theorem, we ensure that all coefficients in~\eqref{eq:VprimeMU} are strictly positive. Then there exists $\theta>0$ such that
	\[
	\frac{d}{dt}\mathcal{V}(t)+2\theta \mathcal{V}(t)\le0.
	\]
	Using~\eqref{eq:V-bounds}, we obtain
	\[
	\mathcal{E}_\zeta(t)\le\kappa \mathcal{E}_\zeta(0)e^{-2\theta t},\qquad\forall t>0,
	\]
showing the claim and the proof of Theorem~\ref{th:MUlyes} is achieved.
	\end{proof}

\subsection{Global stabilization: Compactness-uniqueness method} In this subsection, we establish the global exponential stabilization using the classical compactness–uniqueness method of Lions. Assuming the failure of the observability inequality, we construct a sequence of normalized solutions, so the unique continuation arguments then lead to a contradiction, completing the proof.

\begin{proof}[Proof of Theorem \ref{th:MUes}] First, observe that
\begin{equation}\label{eq:MUobs1}
\begin{split}
T\left\lVert{\zeta_0}\right\rVert_{L^2(\Omega)}^2 
\leq& \left\lVert{\zeta}\right\rVert_{L^2(0,T,L^2(\Omega))}^2 
- \alpha T \int_0^T\int_0^L (\partial_x\zeta(0,y,t))^2\,dy\,dt 
\\&+ \gamma T \int_0^T\int_0^L 
	\left(%
		{ \partial_y\zeta(0,y,t)}
	\right)^2\,dy\,dt\\&+T(2\mu_1+\mu_2) \int_0^T\int_0^L\int_0^L  a(x,y) \zeta^2(x,y,t)\,dx\,dy\,dt \\
&+ T\int_0^T\int_0^L\int_0^L  a(x,y) \mu_2 \zeta^2(x,y,t-h)\,dx\,dy\,dt
\end{split}
\end{equation}
Moreover,  multiplying \eqref{eq:ZKlin1} by $a(x,y)\xi z(x,y,\rho,s)$, integrating in $\Omega\times(0,1)\times(0,T)$ and taking in account that $z(x,y,\rho,t)=\zeta(x,y,t-\rho h)$ we obtain
\begin{equation}\label{eq:MUobs2}
\begin{split}
&\int_0^L\int_0^L\int_0^1 a(x,y)z^2(x,\rho,0)\,d\rho\,dx\,dy \leq \frac{1}{hT}\int_0^T\int_0^L\int_0^L a(x,y) \zeta^2(x,y,t)\,dx\,dy\,dt\\
&+ \left(%
	\frac{1}{Th}+\frac{1}{h}
	\right)
\int_0^T\int_0^L\int_0^L a(x,y) \zeta^2(x,y,t-h)\,dx\,dy\,dt
\end{split}
\end{equation}

As is classical in control theory, Theorem \ref{th:MUes} is a consequence of the following observability inequality
\begin{equation}\label{eq:MUobs}
\begin{aligned}
\mathcal{E}_\zeta(0)  \leq& \mathcal{C}
	\left(
		 \int_0^T\int_0^L (\partial_x\zeta(0,y,t))^2\,dy
		+\int_0^T\int_0^L(\partial_y\zeta(0,y,t))^2\,dy\,dt
	\right.\\
	&
	\left.%
		+\int_0^T\int_{\Omega} a(x,y)(\zeta^2(x,y,t-h)+\zeta^2(x,y,t)\,dx\,dy\,dt%
	\right),
\end{aligned}
\end{equation}
with $C$ a positive constant.

Putting together \eqref{eq:MUobs1} and \eqref{eq:MUobs2}, we see that the observability inequality \eqref{eq:MUobs} holds showing that for any $T$ and $R>0$, there exists $K:=K(R,T)>0$ such that
	\begin{equation}\label{eq:ZKobs}
	\begin{split}
				\|\zeta\|_{L^2(0,T;L^2(\Omega))}^2
			\le&K\Big(
			\int_0^T\int_0^L |\partial_x \zeta(0,y,t)|^2\,dy\,dt
			+ \int_0^T\int_0^L |\partial_y \zeta(0,y,t)|^2\,dy\,dt\\
			&+\int_0^T\int_\Omega a(x,y)\left(\zeta^2(x,y,t)+\zeta^2(x,y,t-h)\right)\,dx\,dy\,dt
			\Big),
		\end{split}
		\end{equation}
is verified for all sufficiently regular solutions of the delayed ZK system
	\[
	\partial_t\zeta+\alpha \partial_x^3\zeta+\gamma \partial_x\partial_y^2\zeta+a(x,y)\zeta+b(x,y)(\xi \zeta+\zeta(t-h))=0,
	\]
	with initial data $\|(\zeta_0,z_0(\cdot,\cdot,-h(\cdot)))\|_{\mathcal H}\le R$.

	\medskip
\noindent\textbf{\textit{Proof of \eqref{eq:ZKobs}}.} Suppose, by contradiction, that \eqref{eq:ZKobs} does not occur. Then there exists a sequence $(\zeta^n)_n\subset\mathcal{B}_X$ of solutions with initial data satisfying $\|(\zeta_0^n,z_0^n(\cdot,\cdot,-h(\cdot)))\|_{\mathcal H}\le R$ such that
	\[
	\lim_{n\to\infty}\frac{\|\zeta^n\|_{L^2(0,T;L^2(\Omega))}^2}{B(\zeta^n)}=+\infty,
	\]
	where
	\begin{equation*}
		\begin{split}
			B(\zeta^n)
			=& \int_0^T\int_0^L|\partial_x\zeta^n(0,y,t)|^2\,dy\,dt
			+\int_0^T\int_0^L|\partial_y\zeta^n(0,y,t)|^2\,dy\,dt\\
			&+\int_0^T\int_\Omega a(x,y)\Big(|\zeta^n(x,y,t)|^2+|\zeta^n(x,y,t-h)|^2\Big)\,dx\,dy\,dt.
		\end{split}
	\end{equation*}
	
	Let $\lambda_n=\|\zeta^n\|_{L^2(0,T;L^2(\Omega))}$ and set
	\[
	v^n(x,y,t)=\frac{\zeta^n(x,y,t)}{\lambda_n}.
	\]
	Then $v^n$ satisfies the normalized system
	\begin{equation*}
		\begin{cases}
			\partial_tv^n+\alpha \partial^3_xv^n+\gamma \partial_x\partial^2_xv^n+a(x,y)v^n+b(x,y)(\xi v^n+v^n(t-h))=0,\\[0.3em]
			v^n(0,y,t)=v^n(L,y,t)=\partial_xv^n(L,y,t)=\partial_yv^n(L,y,t)=0,\\[0.3em]
		         v^n(x,0,t)=v^n(x,L,t)=0,\\[0.3em]
			v^n(x,y,0)=\frac{u_0^n}{\lambda_n}(x,y),\quad v^n(x,y,t)=\frac{z_0^n}{\lambda_n}(x,y,t),\\[0.3em]
			\|v^n\|_{L^2(0,T;L^2(\Omega))}=1,
		\end{cases}
	\end{equation*}
	and $B(v^n)\to0$ as $n\to\infty$.

	From \eqref{eq:MUobs1} and the normalization, we have
	\[
	\|v^n(\cdot,\cdot,t)\|_{L^2(\Omega)}^2
	\le \frac{1}{T}\|v^n\|_{L^2(0,T;L^2(\Omega))}^2 + C\,B(v^n)
	\le \frac{1}{T}+C\,B(v^n),
	\]
	so $\big(v^n(\cdot,\cdot,0)\big)_n$ is bounded in $L^2(\Omega)$.  
	A similar estimate for the delayed term shows that $$\big(\sqrt{a(x,y)}v^n(\cdot,\cdot,-h(\cdot))\big)_n$$ is bounded in $L^2(\Omega\times(0,1))$. Moreover, $(\lambda_n)_n$ is bounded. Thanks to the regularity result (analogous to Proposition \ref{pr:ZKreg}), the sequence $(v^n)_n$ is bounded in $L^2(0,T;H^1(\Omega))$. Using a Gagliardo-Nirenberg inequality, we deduce that $(v^n \partial_xv^n)_n$ is bounded in $L^2(0,T;L^1(\Omega))$.

	Because $v^n\in L^2(0,T;H^1(\Omega))$, we have $\partial_xv^n\in L^2(0,T;L^2(\Omega))$. Using Cauchy–Schwarz,
	\[
	\|\partial_x\partial_y^2v^n\|_{L^2(0,T;H^{-3}(\Omega))}
	\le C\|v^n\|_{L^2(0,T;L^2(\Omega))},
	\]
	and thus $(\partial_x\partial_y^2v^n)_n$ is bounded in $L^2(0,T;H^{-3}(\Omega))$. Consequently, 
	\[
	\partial_tv^n=-\alpha \partial^3_xv^n-\gamma \partial_x\partial^2_yv^n-\lambda_n v^n \partial_xv^n - a(x,y)(\mu_1v^n+\mu_2v^n(t-h))
	\]
	is bounded in $L^2(0,T;H^{-3}(\Omega))$.  
	Hence, by the Aubin–Lions compactness theorem \cite{Simon}, the sequence $(v^n)_n$ is relatively compact in $L^2(0,T;L^2(\Omega))$.  
	Therefore, there exists a subsequence (still denoted $v^n$) such that
$$
		v^n\to v\quad\text{strongly in }L^2(0,T;L^2(\Omega)),
		\qquad \|v\|_{L^2(0,T;L^2(\Omega))}=1.
$$
	By weak lower semicontinuity and the convergence of boundary traces implied by $B(v^n)\to0$, we have
\begin{equation}\label{rrrr-uc}
		v(x,y,t)=0\ \text{in}\ \omega\times(0,T),\qquad
		\partial_xv(0,y,t)=0\ \text{for}\ y\in(0,L),\ t\in(0,T).
\end{equation}
	Moreover, since $(\lambda_n)_n$ is bounded, we can extract a subsequence such that $\lambda_n\to\lambda\ge0$.  
	
	Passing to the limit in the weak formulation of the equation satisfied by $v^n$, we find that $v$ solves
	\begin{equation*}
		\begin{cases}
			\partial_tv+\alpha \partial^3_x v+\gamma \partial_x\partial^2_yv+a(x,y)(\mu_1v+\mu_2v(t-h))+\lambda v \partial_xv=0,\\[0.3em]
			v(0,y,t)=v(L,y,t)= \partial_xv(L,y,t)=\partial_yv(L,y,t)=0,\\[0.3em]
			v(x,0,t)=v(x,L,t)=0,\\[0.3em]
			\|v\|_{L^2(0,T;L^2(\Omega))}=1.
		\end{cases}
	\end{equation*}
If $\lambda = 0$, then $v$ satisfies a linear homogeneous ZK equation with homogeneous boundary conditions, and \eqref{rrrr-uc}. By Holmgren’s uniqueness theorem, this implies that $v\equiv 0$ in $\Omega\times(0,T)$, which contradicts the normalization $\|v\|_{L^{2}(0,T;L^{2}(\Omega))}=1$.  On the other hand, if $\lambda>0$, then $v\in L^{2}(0,T;H^{3}(\Omega))$ and satisfies the homogeneous equation in the whole domain. Applying the unique continuation properties available for the ZK equation (see \cite[Theorem~1.1]{Pa1}, \cite[Theorem~1.1]{Roger2025}, or \cite[Theorem~1.1]{Chen}), we again conclude that $v\equiv 0$ in $\Omega\times(0,T)$, leading to the same contradiction. 

\vspace{0.2cm}

Therefore, \eqref{eq:ZKobs} holds, and consequently, the observability inequality \eqref{eq:MUobs} is also verified. Thus, the theorem is proved. 	
\end{proof}

\section{Conclusion and novelty of the article}\label{Sec4} 

%In this work, we treated the stabilization of the ZK equation in a bounded domain using two different approaches. he first main contribution of this manuscript (Theorem \ref{th:ZKes}) shows that the solutions of system \eqref{eq:ZK} are locally exponentially stabilized whenever the delayed feedback weight is sufficiently small, a result that, remarkably, holds for every length $L>0$, unlike the instantaneous damping models found in the existing literature. Our second and third contributions address the reduced $\mu_i$-system \eqref{eq:MU}, where the dynamics are reformulated through the parameters $\mu_{1}$ and $\mu_{2}$. In this framework, we obtain both local (Theorem \ref{th:MUlyes}) and global (Theorem \ref{th:MUes}) exponential stabilization results, and in the local case, we provide the explicit decay rate that stabilizes the ZK system.

In this work, we established exponential stabilization for the ZK equation on bounded domains through two complementary approaches. First, we proved that system \eqref{eq:ZK} is locally exponentially stabilized for all lengths $L>0$ when the delay weight is sufficiently small. Then, by reformulating the dynamics via the parameters $\mu_1$ and $\mu_2$, we obtained both local and global exponential stabilization for the reduced system \eqref{eq:MU}, including an explicit decay rate in the local case. In contrast with the stabilization mechanisms commonly found in the literature, our problem introduces additional and genuinely new challenges. We summarize below the main aspects that distinguish our setting from the previously cited works.

\subsection{Energy structure under delayed feedback} The state variable is coupled with a distributed delay feedback, so the damping term depends not only on the present state but also on the history of the solution. This brings to the equation a combination of damping and delayed effects, absent in the classical ZK stabilization problems studied so far. As a result, the associated energy functional exhibits a mixed structure involving both present and delayed components, and its time derivative no longer has a definite sign, preventing a straightforward conclusion about the decay of the energy.
\subsection{Stability issues caused by time-delayed} Unlike the results of Dieme, de Moura and Santos \cite{Roger2025} for the n-dimensional ZK equation with localized damping and those of de Moura, Nascimento and Santos \cite{ailton2021} for high-order KP and ZK equations, the presence of delay drastically changes the qualitative behavior of the dynamics: delayed feedback, if not treated with appropriate weights, may destabilize the system or prevent exponential decay of the natural energy. Establishing stability in this setting, therefore, demands new estimates that are not present in previous works.

\subsection{Sharp decay estimates and multidimensional extensions} Compared with the results in \cite{ailton2021} and \cite{Roger2025}, it is important to emphasize that our work provides an explicit decay rate. In both theorems, we obtain optimal local stabilization, giving not only the precise decay rate but also the exact time from which the exponential behavior becomes effective.

Finally, our analysis can be naturally extended to other evolution equations. In particular, thanks to the unique continuation theory developed by Dieme, de Moura, and Santos~\cite{Roger2025}, the methods introduced in our manuscript can be adapted to several multidimensional dispersive models. Among them, we highlight the three-dimensional Kawahara equation
\[
\partial_t v + \partial_x\left( v +  \Delta v - \partial_x^4 v +\frac{1}{2}v^2\right)= 0, 
\qquad (x,y,z,t)\in(0,L)\times(0,L)\times(0,L)\times(0,+\infty),
\]
and the three-dimensional Kadomtsev--Petviashvili type equation
\[
\partial_x\left( \partial_t v + \alpha\,\partial_x^3 v + \beta\,\partial_x^5 v + v\,\partial_x v \right)
+ (\partial_y^2  + \partial_z^2) v = 0,
\qquad (x,y,z,t)\in(0,L)\times(0,L)\times(0,L)\times(0,+\infty),
\]
where $\alpha, \beta \in \mathbb{R}$. For additional background and further developments on these systems, we refer the reader to~\cite{LaSi,SaTz,SaTz1} and the references therein.

\subsection*{Data availability statement} Data sharing does not apply to the current paper as no data were generated or analyzed during this study.

\subsection*{Acknowledgment} This work was carried out during several visits between the authors at the Federal University of Pernambuco and the Federal University of Piauí. The authors gratefully acknowledge the warm hospitality provided by both institutions.


\begin{thebibliography}{100}
\bibitem{Araruna2012}
F.~D. Araruna, R.~A. Capistrano--Filho, and G.~G. Doronin,
\textit{Energy decay for the modified Kawahara equation posed in a bounded domain},
J. Math. Anal. Appl. 385(2), 743--756 (2012).

%\bibitem{besov79}
%O.~V. Besov, V.~P. Il'in, and S.~M. Nikol'skii,
%\textit{Integral Representations of Functions and Imbedding Theorems}, Vol.~I,
%New York--Toronto--London, 1978.


%\bibitem{BLS}
%J.~L. Bona, D. Lannes, and J.-C. Saut,
%\textit{Asymptotic models for internal waves},
%J. Math. Pures Appl. (9) 89(6), 538--566 (2008).

%\bibitem{boumediene2}
%R.~A. Capistrano--Filho, B. Chentouf, L. de Sousa, and V.~H. Gonzalez Martinez,
%\textit{Two stability results for the Kawahara equation with a time-delayed boundary control},
%Z. Angew. Math. Phys. 74:16, 1--26 (2023).

\bibitem{CF-K-P}
R.~A. Capistrano--Filho, V. Komornik, and A.~F. Pazoto,
\textit{Observability of the linear Zakharov--Kuznetsov equation},
Appl. Math. Optim. 91, 45 (2025).

\bibitem{Cerpa2021}
R.~A. Capistrano--Filho, E. Cerpa, and F.~A. Gallego,
\textit{Rapid exponential stabilization of a Boussinesq system of KdV–KdV type},
Commun. Contemp. Math., 25, 03 (2023).

\bibitem{Capistrano2018}
R.~A. Capistrano--Filho and F.~A. Gallego,
\textit{Asymptotic behavior of Boussinesq system of KdV–KdV type},
J. Differential Equations 265(6), 2341--2374 (2018).

\bibitem{martinez2022}
R.~A. Capistrano--Filho and V.~H. Gonzalez Martinez,
\textit{Stabilization results for delayed fifth-order KdV-type equation in a bounded domain},
Mathematical Control and Related Fields 14(1), 284--321 (2024).

\bibitem{Chen}
M. Chen,
\textit{Unique continuation property for the Zakharov–Kuznetsov equation},
Comput. Math. Appl. 77(5), 1273--1281 (2019).

\bibitem{Chen1}
M. Chen, \textit{Global Approximate Controllability of the Korteweg-de Vries Equation by a Finite-Dimensional Force}
Appl Math Optim. 87(12), (2023).

\bibitem{Chen2}
M. Chen, \textit{Approximate controllability and irreducibility of the 2D Zakharov–Kuznetsov–Burgers equation}
ESAIM: Control, Optimisation and Calculus of Variations, 31 (77), (2025).

\bibitem{Chen3}
M. Chen and L. Rosier, 
\textit{Exact controllability of the linear Zakharov–Kuznetsov equation},
Discrete Continuous Dyn.
Syst. B 25,  3889--3916 (2020).

\bibitem{boumediene}
B. Chentouf,
\textit{Well-posedness and exponential stability of the Kawahara equation with a time-delayed localized damping},
Math. Methods Appl. Sci. 45, 10312--10330 (2022).

\bibitem{Pignotti} 
I. Issa and C. Pignotti,
\textit{Time-delayed generalized Korteweg–de Vries-Burgers equation: Well-posedness and exponential decay,}
Nonlinear Analysis: Real World Applications, 89:104519 (2026).

\bibitem{ailton2021}
R.~P. de Moura, A.~C. Nascimento, and G.~N. Santos,
\textit{On the stabilization for the high-order Kadomtsev--Petviashvili and the Zakharov--Kuznetsov equations with localized damping},
Evol. Equ. Control Theory 11, 711--727 (2022).

\bibitem{Roger2025}
R.~P. de Moura, P. Dieme, and G.~N. Santos,
\textit{On the unique continuation and stabilization for the $n$-dimensional Zakharov--Kuznetsov equation},
Preprint (2025).

\bibitem{DoLa} 
G.~G.~Doronin and N.~A. Larkin,
\textit{Stabilization of Regular Solutions for the Zakharov-Kuznetsov Equation Posed on Bounded Rectangles and on a Strip},
Proceedings of the Edinburgh Mathematical Society. 58(3), 661--682 (2015).

\bibitem{DoLa1}
G.~G.~Doronin and N.~A. Larkin,
\textit{Stabilization for the linear Zakharov–Kuznetsov equation without critical size restrictions},
Journal of Mathematical Analysis and Applications 428(1), 337--355 (2015).

\bibitem{Faminskii-95}
A.~V. Faminskii,
\textit{The Cauchy problem for the Zakharov--Kuznetsov equation},
Differ. Equ. 31, 1002--1012 (1995).

\bibitem{Fr}
A.~Friedman, 
\textit{Partial Differential Equations}, 
Dover ed., New York (2008).

\bibitem{Panthee2011}
D.~A. Gomes and M. Panthee,
\textit{Exponential energy decay for the Kadomtsev--Petviashvili (KP-II) equation},
São Paulo J. Math. Sci. 5(2), 135--148 (2011).

\bibitem{Kinoshita-21}
S. Kinoshita,
\textit{Global well-posedness for the Cauchy problem of the Zakharov--Kuznetsov equation in 2D},
Ann. Inst. H. Poincaré Anal. Non Linéaire 38, 451--505 (2021).

\bibitem{KinSch-21}
S. Kinoshita and R. Schippa,
\textit{Loomis–Whitney-type inequalities and low-regularity well-posedness of the periodic Zakharov--Kuznetsov equation},
J. Funct. Anal. 280, 108904 (2021).

%\bibitem{Haragus}
%M. Haragus,
%\textit{Model equations for water waves in the presence of surface tension},
%Eur. J. Mech. B Fluids 15(4), 471--492 (1996).

%\bibitem{Hasimoto1970}
%H. Hasimoto,
%\textit{Water waves},
%Kagaku 40, 401--408 (1970) (in Japanese).

%\bibitem{Karpman}
%V.~I. Karpman,
%\textit{Transverse stability of Kawahara solitons},
%Phys. Rev. E 47(1), 674--676 (1993).

%\bibitem{Kawahara}
%T. Kawahara,
%\textit{Oscillatory solitary waves in dispersive media},
%J. Phys. Soc. Japan 33, 260--264 (1972).

\bibitem{LaLiSa-13}
D. Lannes, F. Linares, and J.-C. Saut,
\textit{The Cauchy problem for the Euler–Poisson system and derivation of the Zakharov--Kuznetsov equation},
in: Cicognani et al. (eds.), 
\textit{Studies in Phase Space Analysis with Applications to PDEs},
Progr. Nonlinear Differ. Equ. Appl. 84, Birkhäuser, 2013.

\bibitem{La}
N. A. Larkin, 
\textit{Global regular solutions for the 3D Zakharov-Kuznetsov equation posed on a bounded domain},
Differential Integral Equations 29:(7/8), 775--790 (2016).

\bibitem{LaSi}
N.~A. Larkin and M. Simões,
\textit{Global regular solutions for the 3D Kawahara equation posed on unbounded domains},
Z. Angew. Math. Phys. 67(4), Art.~98 (2016).

\bibitem{LinPas-09}
F. Linares and A. Pastor,
\textit{Well-posedness for the two-dimensional modified Zakharov--Kuznetsov equation},
SIAM J. Math. Anal. 41, 1323--1339 (2009).

\bibitem{LinPanRo-19}
F. Linares, M. Panthee, and N. Tzvetkov,
\textit{On the periodic Zakharov--Kuznetsov equation},
Discrete Contin. Dyn. Syst. 39, 3521--3533 (2019).

\bibitem{Lions}
J.-L. Lions,
\textit{Exact controllability, stabilization and perturbations for distributed systems},
SIAM Rev. 30(1), 1--68 (1988).

\bibitem{Lions1988}
J.-L. Lions,
\textit{Contrôlabilité Exacte, Perturbations et Stabilisation de Systèmes Distribués}, Vol.~22,
Elsevier--Masson, 1990.

\bibitem{Zuazua2002}
G. Menzala, C. Vasconcellos, and E. Zuazua,
\textit{Stabilization of the Korteweg–de Vries equation with localized damping},
Quart. Appl. Math. LX, 111--129.

\bibitem{MolPil-15}
L. Molinet and D. Pilod,
\textit{Bilinear Strichartz estimates for the Zakharov--Kuznetsov equation and applications},
Ann. Inst. H. Poincaré Anal. Non Linéaire 32, 347--371 (2015).

\bibitem{MoPa-99}
S. Monro and E.~J. Parkes,
\textit{The derivation of a modified Zakharov--Kuznetsov equation and the stability of its solutions},
J. Plasma Phys. 62, 305--317 (1999).

\bibitem{MoPa-00}
S. Monro and E.~J. Parkes,
\textit{Stability of solitary-wave solutions to a modified Zakharov--Kuznetsov equation},
J. Plasma Phys. 64, 411--426 (2000).

\bibitem{Osawa-22}
S. Osawa,
\textit{Local well-posedness for the Zakharov--Kuznetsov equation in Sobolev spaces},
arXiv:2207.04657 [math.AP] (2022).

\bibitem{Pa1}
M. Panthee,
\textit{A note on the unique continuation property for Zakharov–Kuznetsov equation},
Nonlinear Anal. 59, 425--438 (2004).

\bibitem{Pazoto2008}
A.~F. Pazoto and L. Rosier,
\textit{Stabilization of a Boussinesq system of KdV–KdV type},
Syst. Control Lett. 57(8), 595--601 (2008).

%\bibitem{kp70}
%B.~B. Kadomtsev and V.~I. Petviashvili,
%\textit{On the stability of solitary waves in weakly dispersive media},
%Sov. Phys. Dokl. 15, 539--549 (1970).

\bibitem{SaTz}
J.-C. Saut and B. Scheurer,
\textit{Unique continuation for some evolution equations},
J. Differential Equations 66, 118--139 (1987).

\bibitem{SaTz1}
J.-C. Saut and N. Tzvetkov,
\textit{The Cauchy problem for higher-order KP equations},
J. Differential Equations 153(1), 196--222 (1999).

\bibitem{Simon}
J. Simon,
\textit{Compact sets in the space $L^p(0,T;B)$},
Ann. Mat. Pura Appl. CLXVI(4), 65--96 (1987).

\bibitem{Valein}
J. Valein,
\textit{On the asymptotic stability of the Korteweg–de Vries equation with time-delayed internal feedback},
Math. Control Relat. Fields 12(3), 667--694 (2022).

\bibitem{ZK1972}
V.~E. Zakharov and E.~A. Kuznetsov,
\textit{Turbulence of Ion Sound in a Plasma Located in a Magnetic Field},
Sov. Phys. JETP 35:2, 584--492 (1972).

\bibitem{ZK1974}
V.~E. Zakharov and E.~A. Kuznetsov,
\textit{Three-dimensional solitons},
Sov. Phys. JETP 39:2, 285--286 (1974).

\end{thebibliography}
\end{document}